\documentclass[a4paper]{amsart}

\usepackage{amssymb}
\usepackage{xy} \xyoption{all}
\usepackage{aliascnt}
\usepackage{hyperref}
\usepackage{xcolor}

\DeclareMathOperator{\range}{Im}  

\newcommand{\andSep}{\,\,\,\text{ and }\,\,\,}
\newcommand{\ZZ}{{\mathbb{Z}}}
\newcommand{\CC}{{\mathbb{C}}}
\newcommand{\ca}{$C^*$-algebra}
\newcommand{\tensMax}{\otimes_{\mathrm{max}}}
\newcommand{\tensMin}{\otimes}
\newcommand{\tensProj}{\widehat{\otimes}}
\newcommand{\tensInj}{\check{\otimes}}
\newcommand{\red}{{\mathrm{red}}}

\numberwithin{equation}{section}

\newtheorem{lma}{Lemma}[section]

\newaliascnt{thmCt}{lma}
\newtheorem{thm}[thmCt]{Theorem}
\aliascntresetthe{thmCt}

\newaliascnt{corCt}{lma}
\newtheorem{cor}[corCt]{Corollary}
\aliascntresetthe{corCt}

\newaliascnt{prpCt}{lma}
\newtheorem{prp}[prpCt]{Proposition}
\aliascntresetthe{prpCt}

\theoremstyle{definition}

\newaliascnt{pgrCt}{lma}
\newtheorem{pgr}[pgrCt]{Paragraph}
\aliascntresetthe{pgrCt}

\newaliascnt{dfnCt}{lma}
\newtheorem{dfn}[dfnCt]{Definition}
\aliascntresetthe{dfnCt}

\newaliascnt{rmkCt}{lma}
\newtheorem{rmk}[rmkCt]{Remark}
\aliascntresetthe{rmkCt}

\newaliascnt{qstCt}{lma}
\newtheorem{qst}[qstCt]{Question}
\aliascntresetthe{qstCt}

\newaliascnt{cnvCt}{lma}
\newtheorem{cnv}[cnvCt]{Convention}
\aliascntresetthe{cnvCt}

\newaliascnt{exaCt}{lma}
\newtheorem{exa}[exaCt]{Example}
\aliascntresetthe{exaCt}

\newcounter{theoremintro}

\newaliascnt{thmIntroCt}{theoremintro}
\newtheorem{thmIntro}[thmIntroCt]{Theorem}
\aliascntresetthe{thmIntroCt}

\newaliascnt{qstIntroCt}{theoremintro}
\newtheorem{qstIntro}[qstIntroCt]{Question}
\aliascntresetthe{qstIntroCt}

\newaliascnt{dfnIntroCt}{theoremintro}
\newtheorem{dfnIntro}[dfnIntroCt]{Definition}
\aliascntresetthe{dfnIntroCt}

\newaliascnt{prpIntroCt}{theoremintro}
\newtheorem{prpIntro}[prpIntroCt]{Proposition}
\aliascntresetthe{prpIntroCt}

\title{The ideal separation property for reduced group $C^*$-algebras}

\author{Are Austad}
\address{Are~Austad, Department of Mathematics, University of Oslo, P.O. Box 1053 Blindern, N-0316 Oslo, Norway}
\email{areaus@math.uio.no}
\urladdr{www.sites.google.com/view/austad}

\author{Hannes Thiel}
\address{Hannes~Thiel, Department of Mathematical Sciences, Chalmers University of Technology and University of Gothenburg, Gothenburg SE-412 96, Sweden.}
\email{hannes.thiel@chalmers.se}
\urladdr{www.hannesthiel.org}

\subjclass[2020]%
{Primary
22D15, 
43A20, 
46H10; 
Secondary
22D20, 
22D25, 
46L05. 
}


\thanks{The first named author was partially supported by the Independent Research Fund Denmark
	through grant number 1026-00371B and The Research Council of Norway project 324944. 
	The second named author was partially supported by the Knut and Alice Wallenberg Foundation (KAW 2021.0140).
}

\date{\today}

\begin{document}
	

\begin{abstract}
We say that an inclusion of an algebra $A$ into a $C^*$-algebra $B$ has the ideal separation property if closed ideals in $B$ can be recovered by their intersection with $A$.
Such inclusions have attractive properties from the point of view of harmonic analysis and noncommutative geometry. 
We establish several permanence properties of locally compact groups for which $L^1(G) \subseteq C^*_{\mathrm{red}}(G)$ has the ideal separation property.
\end{abstract}
	
\maketitle
	
\section{Introduction}

A locally compact group $G$ is called \emph{$*$-regular} if for any closed, two-sided ideal~$I$ in the full group \ca{} $C^*(G)$, we have $I = \overline{I \cap L^1(G)}$, where the closure is in $C^*(G)$-norm.
An in-depth study of groups with this property was first initiated by Boidol, Leptin, Schurman and Vahle in \cite{BoiLepSchVah78PrimGpCAlg}.
They identified that groups of polynomial growth are $*$-regular.  
Boidol then continued to add to the list of $*$-regular groups, showing that semidirect products of abelian groups 
belong to this class \cite{Boi82StarRegSomeSolvGps}.
He also characterized the connected groups which are $*$-regular \cite{Boi82CtdGpsPolyDual}.
Later, Hauenschild, Kaniuth and Voigt \cite{HauKanVoi90StarRegTensProdCAlg} showed that the class of $*$-regular groups is closed under Cartesian products. 
Boidol conjectured that every symmetric (=hermitian) locally compact group is $*$-regular, and that conversely every almost connected, $*$-regular group is symmetric \cite{Boi82StarRegSomeSolvGps}.
To honor his contributions to the theory, $*$-regular groups are sometimes referred to as Boidol groups 
\cite[Definition~10.5.8]{Pal01BAlg2}.
	
A closely related notion is that of $C^*$-uniqueness, where a locally compact group~$G$ is said to be \emph{$C^*$-unique} if the Banach $*$-algebra $L^1(G)$ has a unique $C^*$-norm. 
It is not difficult to see that $L^1(G)$ has a unique $C^*$-norm if and only if, for any non-zero, closed, two-sided ideal $I$ in $C^*(G)$, we have $I \cap L^1(G) \neq \{0\}$. 
From this it follows that every $*$-regular group is $C^*$-unique, and that every $C^*$-unique group has the weak containment property (the full and reduced group \ca{} coincide), and therefore is amenable;
see also \cite[Proposition~2]{Boi84GpAlgsUniqueCNorm}.
In recent work, the first named author and Raum \cite{AusRau23arX:DetectIdlsRedCrProd} showed that locally virtually polycyclic groups are $C^*$-unique.

Boidol gave examples of amenable groups which are not $C^*$-unique, and of $C^*$-unique groups which are not $*$-regular \cite[p.230]{Boi84GpAlgsUniqueCNorm}, thereby showing strict containment between these classes of groups.
However, it remains unknown if such examples exist among discrete groups, prompting Leung and Ng to pose the question if all discrete, amenable group are $C^*$-unique \cite{LeuNg04PermPropCUniqueGps}.
An even more ambitious, yet also unresolved, question is:

\begin{qstIntro}
\label{qst:discrete-amenable-is-*-regular}
\emph{Are all discrete, amenable groups $*$-regular?}
\end{qstIntro}

While Boidol focused on $*$-regularity and uniqueness of $C^*$-norm for $L^1$-algebras of locally compact groups, the notions were studied for more general $*$-algebras by Barnes in \cite{Bar83StarRegUniqueCNorm}. 
We also mention that $*$-regularity of more general convolution algebras has been studied by Leung and Ng \cite{LeuNg06FctlIdStarReg1, LeuNg06FctlIdStarReg2}, Leung \cite{Leu11RegGenGpAlg}, and more recently by Flores \cite{Flo24arX:PolyGrwthFctlCalc}.
On the other hand, $C^*$-uniqueness was studied for twisted group algebras by the first named author \cite{Aus21SpecInvTwConvAlg}, and for twisted groupoid algebras in collaboration with Ortega \cite{AusOrt22CUniuqueGpds}.

A closely related topic is \emph{algebraic $C^*$-uniqueness} for a discrete group $\Gamma$, defined by uniqueness of the $C^*$-norm on the group algebra $\CC[\Gamma]$.
It is straightforward to see that this implies uniqueness of the $C^*$-norm on $\ell^1(\Gamma)$, and consequently, it also implies amenability.
However, examples like the group of integers show that $C^*$-uniqueness is a strictly weaker property than algebraic $C^*$-uniqueness.
The question of which discrete groups are algebraically $C^*$-unique has attracted much attention in recent years \cite{GriMusRor18JustInfCAlgs, AleKye19UniqueCNormGpRg, Sca20TorsFreeAlgCUniqueGp}.
Interestingly, $*$-regularity of~$\CC[\Gamma]$ appears to not have been investigated so far.
	
\medskip

One can view $*$-regularity and $C^*$-uniqueness of a locally compact group $G$ as the ideal separation and ideal intersection properties for the inclusion $L^1(G) \subseteq C^*(G)$.
Here, we say that a subalgebra $A$ of a \ca{} $B$ has the \emph{ideal separation property} (ISP for short) if for any closed, two-sided ideals $I$ and $J$ in $B$ we have $I \subseteq J$ whenever $I \cap A \subseteq J \cap A$.
Likewise, we say $A \subseteq B$ has the \emph{ideal intersection property} (IIP for short) if for any non-zero, closed, two-sided ideal $I$ in $B$ we have $I \cap A \neq \{0\}$;
see \autoref{dfn:ISP-IIP}.

Barnes' notions of $*$-regularity and $C^*$-uniqueness for a $*$-algebra $A$ then correspond precisely to the ISP and the IIP for the inclusion $A \subseteq C^*(A)$, where $C^*(A)$ is the enveloping \ca{} of $A$.
(One needs to assume that $A$ admits some $C^*$-norm so that $A$ embeds into $C^*(A)$.)

The ISP and the IIP have also been extensively studied for \ca{s} arising from dynamical systems or groupoids.
For example, Kawamura and Tomiyama \cite{KawTom90PropTopDynSysCAlg} showed that the action of an amenable, discrete group $\Gamma$ on a compact, Hausdorff space~$X$ is topologically free if and only if $C(X) \subseteq C(X) \rtimes \Gamma$ has the ISP.
For actions on possibly noncommutative \ca{s} $A$, the ISP and the IIP for the inclusion $A \subseteq A \rtimes_\red \Gamma$ have been studied for example by Sierakowski \cite{Sie10IdealsRedCrProd}, Kennedy and Schafhauser \cite{KenSch19NCBoundariesIdealRedCrProd} and Geffen and Ursu \cite{GefUrs23arX:SimpleCrProdFCHyperGps}.
The IIP and the ISP for the inclusion $C(\mathcal{G}^{(0)}) \subseteq C^*_\red(\mathcal{G})$ for a groupoid $\mathcal{G}$ have for example been studied by Kennedy, Kim, Li, Raum and Ursu \cite{KenKimLiRauUrs21arX:ISPEssGpd} and Brix, Carlsen and Sims \cite{BriCarSim24IdlStrEtaleGpd}.

We note that the ISP and the IIP are always studied with respect to the \emph{reduced} crossed product and the \emph{reduced} groupoid algebra.
The reason is that the kernels of the quotient maps $A \rtimes G \to A \rtimes_\red G$ and $C^*(\mathcal{G}) \to C^*_\red(\mathcal{G})$ have trivial intersection with $A$ and $C(\mathcal{G}^{(0)})$, respectively.

Similarly, for the canonical quotient map $\lambda\colon C^*(G) \to C^*_\red(G)$, we always have $L^1(G) \cap \ker \lambda = \{0\}$.   
Therefore, if $G$ is nonamenable, then $L^1(G) \subseteq C^*(G)$ does not have the IIP, let alone the ISP.
We therefore believe that the natural setting to study ideal separation and ideal intersection for a locally compact group $G$ is with respect to the inclusion of $L^1(G)$ in the reduced $C^*$-algebra $C^*_\red(G)$, and we define:

\begin{dfnIntro}
\label{dfn:ISP-IIP-groups}
We say that a locally compact group $G$ has
\begin{enumerate}
\item 
the \emph{ideal separation property}, ISP for short, if for any distinct closed, two-sided ideals $I,J \subseteq C^*_\red(G)$ we have $I \cap L^1(G) \neq J \cap L^1(G)$.
\item 
the \emph{ideal intersection property}, IIP for short, if for every non-zero, closed, two-sided ideal $I \subseteq C^*_\red(G)$ we have $I \cap L^1(G) \neq \{0\}$.
\end{enumerate}
\end{dfnIntro}

We make the following easy, yet useful, observation, whose proof follows from the above discussion:

\begin{prpIntro}
\label{prp:CharStarReg}
Let $G$ be a locally compact group.
The following statements hold:
\begin{enumerate}
\item 
The group $G$ is $*$-regular if and only if $G$ is amenable and has the ISP.
\item 
The group $G$ is $C^*$-unique if and only if $G$ is amenable and has the IIP.
\end{enumerate}
\end{prpIntro}

Thus, the known classes of $*$-regular or $C^*$-unique groups provide examples of (automatically amenable) groups with the ISP or the IIP.
Moreover, there are many non-amenable groups with the ISP or the IIP.
For example, it is immediate that the class of groups with the ISP includes all $C^*$-simple groups, that is, all groups $G$ for which $C^*_\red(G)$ is simple.

We note that \autoref{dfn:ISP-IIP-groups} is inspired by the analogous property for $C^*$-dynamical systems from \cite{AusRau23arX:DetectIdlsRedCrProd}, where it is termed the $L^1$-IIP.

\medskip

In this paper we establish several permanence properties of the ISP and the IIP for locally compact groups, as summarized in the following theorem.

\begin{thmIntro}
\label{prp:PermanenceSummary}
Let $G$ and $H$ be locally compact groups. 
\begin{enumerate}
\item 
If $G$ has the ISP (the IIP) and $N \subseteq G$ is a closed (compact), amenable, normal subgroup, then $G/N$ has the ISP (the IIP);
see \autoref{prp:Quotient}.
\item
If every open, relatively compact subset of $G$ is contained in a closed subgroup with the ISP (the IIP), then $G$ has the ISP (the IIP);
see \autoref{prp:ApproxByClosedSubgroups}.
\item 
If $G = \varprojlim_\lambda G_\lambda$ is an inverse limit such that the maps $G \to G_\lambda$ are surjective with compact kernel, then $G$ has the ISP (the IIP) if and only if each $G_\lambda$ does;
see \autoref{prp:InvLim}.
\item 
The product group $G \times H$ has the IIP (the ISP) if and only if both $G$ and $H$ have the IIP (the ISP, and the pair of group \ca{s} $C^*_\red(G)$ and $C^*_\red(H)$ has Tomiyama's property~(F));
see \autoref{prp:ProductGp}.
\end{enumerate}
\end{thmIntro}
	
Statement~(2) shows in particular that the ISP and the IIP are local properties;
see \autoref{prp:LocalProps}.
Further, statement~(3) implies that the ISP and the IIP for an almost connected, locally compact group are determined by its Lie group quotients;
see \autoref{prp:LieQuotients}.

We note that our results for Cartesian products allow us to deduce that groups such as $\ZZ \times \mathbb{F}_2$ have the ISP, which was not known with the current technology. 
Further, \autoref{prp:ProductGp} offers a strategy to exhibit a counterexample to the following natural extension of \autoref{qst:discrete-amenable-is-*-regular} to the setting of nonamenable groups:

\begin{qstIntro}
\label{qst:DiscreteISP}
\emph{Do all discrete groups have the ideal separation property?}
\end{qstIntro}

Indeed, if $\Gamma$ and $\Delta$ are discrete groups for which the pair of group \ca{s} $C^*_\red(\Gamma)$ and $C^*_\red(\Delta)$ does not have Tomiyama's property (F), then $\Gamma \times \Delta$ does not have the ISP.
It remains unclear if such a pair of groups exists;
see \autoref{qst:PropFGroupAlgs} and the discussion preceding it.

\medskip
	
The paper is structured as follows. 
In \autoref{sec:defn-and-first-props} we introduce the ISP and the IIP for an inclusion of a subalgebra in a \ca{}.
We give characterizations in terms of primitive ideals and $*$-structure spaces.

Specializing to locally compact groups, we treat the inclusion $L^1(G) \subseteq C^*_\red(G)$ in detail in \autoref{sec:perm}, and we investigate preservation of the ISP and the IIP with respect to quotients, approximation by closed subgroups, inductive limits, and inverse limits.
Lastly, \autoref{sec:TensProd} deals with the ISP and the IIP for the inclusion of the algebraic and projective tensor product in the minimal tensor product of \ca{s}. 
This is then applied to characterize the ISP and the IIP for Cartesian products of locally compact groups. 
	
\subsection*{Acknowledgments}
	
The authors thank Leonel Robert for valuable comments on the minimal tensor product of \ca{s}.

\section{Definition and first properties}
\label{sec:defn-and-first-props}

In this section, we introduce the ideal separation property (ISP) and the ideal intersection property (IIP) for the inclusion of a subalgebra into a \ca{}, and we study basic properties of these notions.
We neither assume that the subalgebra is dense nor closed, to allow for the study of inclusions such as $A \subseteq A\rtimes_\red G$ as well as $L^1(G) \subseteq C^*_\red(G)$.

Later we will specialize to the inclusion $A \subseteq B$ of a dense $*$-subalgebra $A$ in a \ca{} $B$, and we give characterizations of the ISP (\autoref{prp:CharISPDense}) and the IIP (\autoref{prp:CharIIP}) in terms of the connection between the $*$-structure space of $A$ and the primitive ideal space of $B$.
We also characterize the ISP and the IIP using estimates on $C^*$-seminorms;
see \autoref{prp:norm-reformulations}.

\begin{cnv}
All ideals in this paper are assumed to be two-sided. 
\end{cnv}

\begin{dfn}
\label{dfn:ISP-IIP}
Let $A \subseteq B$ be an inclusion of a subalgebra into a \ca{}~$B$.
\begin{enumerate}
\item 
We say that $A \subseteq B$ has the \emph{ideal separation property} (ISP for short), or that \emph{$A$ separates ideals in $B$}, if for all distinct closed ideals $I$ and $J$ in $B$ we have $I \cap A \neq J \cap A$.
\item 
We say that $A \subseteq B$ has the \emph{ideal intersection property} (IIP for short), or that \emph{$A$ detects ideals in $B$}, if for every non-zero, closed ideal $I$ in $B$ we have $I \cap A \neq \{0\}$. 
\end{enumerate}
\end{dfn}

The next result provides basic characterizations of the ISP.
Most useful is condition~(3), which states that every closed ideal in $B$ can be \emph{reconstructed} from its intersection with $A$.

\begin{prp}
\label{prp:CharISPGeneral}
Let $A \subseteq B$ be an inclusion of a subalgebra into a \ca{}~$B$.
The following statements are equivalent:
\begin{enumerate}
\item 
The inclusion $A \subseteq B$ has the ISP.
\item
For every closed ideal $I \subseteq B$, the inclusion $(A + I)/I \subseteq B/I$ has the IIP.
\item
Every closed ideal $I \subseteq B$ is equal to the closed ideal of $B$ generated by $I \cap A$.
\item
For any closed ideals $I,J \subseteq B$ we have $I \subseteq J$ whenever $I \cap A \subseteq J \cap A$.
\end{enumerate}
\end{prp}
\begin{proof}
We show the implications `$(1) \Rightarrow (2) \Rightarrow (3) \Rightarrow (4) \Rightarrow (1)$'.
Assuming~(1), let us verify~(2).
Let $I \subseteq B$ be a closed ideal, and let $J \subseteq B/I$ be a non-zero, closed ideal.
Let $\pi_I \colon B \to B/I$ denote the quotient map, and set $J' := \pi_I^{-1}(J)$, which is a closed ideal of $B$ with {$I \subseteq J'$ and $I \neq J'$}. 
Since $A$ separates ideals in $B$, we obtain $a \in A$ with {$a \in J'$} 
but $a \notin I$.
Then $0 \neq \pi_I(a) \in J \cap \pi_I(A)$.

Assuming~(2), let us verify~(3).
Let $I \subseteq B$ be a closed ideal, and let $H$ denote the closed ideal of $B$ generated by $I \cap A$.
We clearly have $H \subseteq I$, and we need to show that $I \subseteq H$.
Let $\pi_H \colon B \to B/H$ denote the quotient map.
We have $I \cap A \subseteq H$, { and using $H \subseteq I$ a straightforward calculation will show that $\pi_H (I) \cap \pi_H ( A) = \pi_H (I \cap A) = \{0\}$.} 
Using that $\pi_H(A) \subseteq B/I$ has the IIP, we get $\pi_H(I)=\{0\}$, and so $I \subseteq H$.

Assuming~(3), let us verify~(4).
Let $I,J \subseteq B$ be closed ideals with $I \cap A \subseteq J \cap A$.
Then the closed ideal of $B$ generated by $I \cap A$ is contained in the closed ideal generated by $J \cap A$, and by assumption this implies that $I \subseteq J$.

Assuming~(4), let us verify~(1).
To show that $A$ separates ideals in $B$, let $I,J \subseteq B$ be closed ideals with $I \cap A = J \cap A$.
We need to show that $I = J$.
But the assumption immediately implies that $I \subseteq J$ and $J \subseteq I$.
\end{proof}

It follows easily from \autoref{dfn:ISP-IIP} that both the ISP and the IIP pass to superalgebras in the following sense.

\begin{prp}
\label{prp:Superalgebras}
Let $B$ be a \ca{}, and let $A_1 \subseteq A_2 \subseteq B$ be subalgebras.
If $A_1 \subseteq B$ has the ISP (the IIP), so does $A_2 \subseteq B$.                      
\end{prp}

We now turn to characterizations of the ISP and the IIP in terms of primitive ideal spaces and $*$-structure spaces. 
We recall the basic definitions and refer the reader to \cite[Chapter~10.5]{Pal01BAlg2} for further details.

For a $*$-algebra $A$ we denote by $\Pi_A^*$ the set of all kernels of topologically irreducible $*$-representations of $A$, and we equip it with the hull-kernel topology. 
This topological space is known as the \emph{$*$-structure space} of $A$. 
For any subset $S$ of $A$ the \emph{hull} of $S$ is defined as
\begin{equation*}
h(S) = \big\{ P \in \Pi_A^* \mid S \subseteq P \big\}.
\end{equation*}
For any subset $L$ of $\Pi_A^*$ the \emph{kernel} of $L$ is defined as
\begin{equation*}
k(L) = \bigcap_{Q \in L} Q \subseteq A.
\end{equation*}
We also write $h_A(S)$ for $h(S)$, and $k_A(L)$ for $k(L)$.
As the space $\Pi_A^*$ is equipped with the \emph{hull-kernel topology}, a subset $K \subseteq \Pi_A^*$ is defined to be closed if $K=h(k(K))$.
For every subset $S \subseteq A$, the set $h(S)$ is closed.

Under mild assumptions on $A$ (that are always satisfied for Banach $*$-algebras, in particular for~$L^1(G)$ for a locally compact group $G$), whenever $I \subseteq A$ is a $*$-ideal, then there is a natural homeomorphism of $\Pi_A^* \setminus h(I)$ with $\Pi_I^*$ given by the map $P \mapsto P \cap I$;
see \cite[Theorem~10.5.5]{Pal01BAlg2}.

If $A$ is a \ca{}, then the $*$-structure space coincides with the primitive ideal space, which we denote by $\Pi_A$. 
Moreover, in a \ca{} any closed ideal is the intersection of the primitive ideals containing it, so in this case we have $I = k(h(I))$ for every closed ideal $I \subseteq A$.

\medskip

Let $A \subseteq B$ be an inclusion of a dense $*$-subalgebra $A$ into a \ca{} $B$.
We will characterize the ISP and the IIP by conditions that capture to what extent~$A$ separates the primitive ideal space of $B$.
Since $A$ is dense in $B$, the restriction of a topologically irreducible $*$-representation of $B$ to $A$ is again topologically irreducible.
We therefore have a well-defined map between the $*$-structure spaces
\begin{equation}
\label{eq:structure-space-map}
\Psi \colon \Pi_B \to \Pi^*_A, \quad P \mapsto P \cap A.
\end{equation}
Given a subset $S \subseteq A$, we have
\[
\Psi^{-1}(h_A(S)) = h_B(S),
\]
which shows that $\Psi$ is continuous.
We show in \autoref{prp:CharISPDense} that~$\Psi$ is open if and only if $A \subseteq B$ has the ISP.

\begin{lma}
\label{prp:Pre-CharISP}
Let $A \subseteq B$ be an inclusion of a dense $*$-subalgebra into a \ca{}.
Let $Y \subseteq \Pi_B$, and let $P \in \Pi_B$.
Then $\Psi(P)$ belongs to the closure of $\Psi(Y)$ in $\Pi_A^*$ if and only if $k_B(Y) \cap A \subseteq P \cap A$.
\end{lma}
\begin{proof}
We first note that $\Psi(P)$ belongs to the closure of $\Psi(Y)$ in $\Pi_A^*$ if and only if $\Psi(P) \in h_A(k_A(\Psi(Y))$, that is, if and only if 
\[
k_A(\Psi(Y)) \subseteq \Psi(P) = P \cap A.
\]
We have
\[
k_B(Y) \cap A 
= \Big( \bigcap \big\{ Q : Q \in Y \big\} \Big) \cap A
= \bigcap \big\{ Q \cap A : Q \in Y \big\},
= k_A(\Psi(Y))
\]
which completes the proof.
\end{proof}

\begin{prp}
\label{prp:CharISPDense}
Let $A \subseteq B$ be an inclusion of a dense $*$-subalgebra $A$ into a \ca{} $B$.
The following statements are equivalent:
\begin{enumerate}
\item 
The inclusion $A \subseteq B$ has the ISP.
\item 
We have $I = \overline{I \cap A}^{\Vert \cdot \Vert_B}$ for every closed ideal $I \subseteq B$. 
\item 
The map $\Psi$ of \eqref{eq:structure-space-map} is a homeomorphism onto its image.
\item 
The map $\Psi$ of \eqref{eq:structure-space-map} is open onto its image.
\item 
For every closed and proper subset $K \subseteq \Pi_B$ and every $P \in \Pi_B \setminus K$ there exists $f \in A$ such that $f \in k_B(K)$ and $f \notin P$.
\end{enumerate}
\end{prp}
\begin{proof}
Assuming~(1), let us verify~(2).
Let $I \subseteq B$ be a closed ideal. 
Since $A \subseteq B$ is dense, the closed ideal generated by $I \cap A$ is the closure of $I \cap A$.
Therefore  $I = \overline{I \cap A}$ by \autoref{prp:CharISPGeneral}(3).

\medskip		

Assuming~(2), let us verify~(3).
Observing that $P = \overline{P \cap A}$ for all $P \in \Pi_B$ we deduce that $\Psi$ is injective.
Since $\Psi$ is continuous, it remains to show that $\Psi$ maps closed subsets of $\Pi_B$ to relatively closed subsets of $\Psi(\Pi_B)$ in $\Pi^*_A$.
Let $K \subseteq \Pi_B$ be a closed subset.
To show that $\Psi(K)$ is closed inside the image of $\Psi$, let $P \in \Pi_B$ such that $\Psi(P)$ belongs to the closure of $\Psi(K)$ in $\Pi_A^*$.
Using that assumption at the first and last step, and using \autoref{prp:Pre-CharISP} at the second step, we get
\[
k_B(K)
= \overline{k_B(K) \cap A} 
\subseteq \overline{P \cap A}
= P
\]
Thus, $P \in h_B(k_B(K))$, which implies $P \in K$ since $K$ is closed.
This shows that~$\Psi(K)$ is closed relative to the image of $\Psi$.
		
\medskip		

It is clear that~(3) implies~(4).
Assuming~(4), let us verify~(5).
Suppose that $K \subseteq \Pi_B$ is a proper, closed subset, and let $P \in \Pi_B \setminus K$.
By assumption, $\Psi(K)$ is closed relative to the image of $\Psi$, and thus $\Psi(P) \notin h_A(k_A(\Psi(K)))$.
This means that $k_A(\Psi(K))$ is not a subset of $\Psi(P)$.
Choose $f \in k_A(\Psi(K))$ with $f \notin \Psi(P)$.
We have seen in the proof of \autoref{prp:Pre-CharISP} that $k_A(\Psi(K))=k_B(K) \cap A$.
Thus, we have $f \in A$ and $f\in k_B(K)$, but $f \notin \Psi(B) = P\cap A$, which implies that $f \notin P$.

\medskip		

Assuming~(5), let us verify~(1).
Let $I$ and $J$ be distinct closed ideals in $B$.
We need to show that $I \cap A \neq J \cap A$.
Without loss of generality, we may assume that $I \nsubseteq J$. 
Then $h_B(I) \nsubseteq h_B(J)$, which allows us to pick $P \in h_B(I) \setminus h_B(J)$. 
Since $h_B(J)$ is a closed subset of $\Pi_B$, we can apply the assumption to find $f \in A$ with $f \in k_B(h_B(J)) = J$ but $f \notin P$. 
Since $I \subseteq P$, we have $f \notin I$, and thus $I \cap A \neq J \cap A$. 
\end{proof}

\begin{lma}
\label{prp:Pre-CharIIP}
Let $A \subseteq B$ be an inclusion of a dense $*$-subalgebra into a \ca{}, and let $Y \subseteq \Pi_B$.
The following statements hold:
\begin{enumerate}
\item
The subset $Y \subseteq \Pi_B$ is dense if and only if $k_B(Y)=\{0\}$.
\item
The subset $\Psi(Y) \subseteq \range(\Psi)$ is dense if and only if $k_B(Y) \cap A = \{0\}$.
\end{enumerate}
\end{lma}
\begin{proof}
(1)
The closure of $Y$ in $\Pi_B$ is $h_B(k_B(Y))$.
Thus, $Y \subseteq \Pi_B$ is dense if and only if $h_B(k_B(Y))=\Pi_B$.
Since $B$ is a \ca{}, we have $k_B(\Pi_B)=\{0\}$.
It follows that $h_B(k_B(Y))=\Pi_B$ if and only if $k_B(Y) = \{0\}$.

(2)
Given $P \in \Pi_B$, the element $\Psi(P)$ belongs to the closure of $\Psi(Y)$ if and only $k_B(Y) \cap A \subseteq P \cap A$;
see \autoref{prp:Pre-CharISP}.
Thus, $\Psi(Y) \subseteq \range(\Psi)$ is dense if and only if $k_B(Y) \cap A \subseteq P \cap A$ for every $P \in \Pi_B$, which in turn is equivalent to $k_B(Y) \cap A = \{0\}$ since $\bigcap\{P : P \in \Pi_B\} = \{0\}$.
\end{proof}

\begin{prp}
\label{prp:CharIIP}
Let $A \subseteq B$ be an inclusion of a dense $*$-subalgebra $A$ into a \ca{} $B$.
The following statements are equivalent:
\begin{enumerate}
\item 
The inclusion $A \subseteq B$ has the IIP.
\item
A subset $Y \subseteq \Pi_B$ is dense if (and only if) $\Psi(Y) \subseteq \range(\Psi)$ is dense.
\item 
For every closed and proper subset $K \subseteq \Pi_B$ there exists a non-zero $f \in A$ such that $f \in k_B(K)$. 
\end{enumerate}
\end{prp}
\begin{proof}
Assuming~(1), let us verify~(2).
Let $Y \subseteq \Pi_B$.
In general, if $Y$ is dense in $\Pi_B$, then $\Psi(Y)$ is dense in the image of $\Psi$.
To show the converse, assume that $\Psi(Y)$ is dense in $\range(\Psi)$.
Applying \autoref{prp:Pre-CharIIP}, we get $k_B(Y) \cap A = \{0\}$.
Using that $A \subseteq B$ has the IIP, we deduce that $k_B(Y)=\{0\}$.
Applying \autoref{prp:Pre-CharIIP} again, we get that~$Y$ is dense in $\Pi_B$.

Assuming~(2), let us verify~(3).
Let $K \subseteq \Pi_B$ be a closed and proper subset. 
The assumption implies that $\Psi(K)$ is not dense in $\range(\Psi)$.
By \autoref{prp:Pre-CharIIP} we get $k_B(K) \cap A \neq \{0\}$, which shows that there exists a non-zero $f \in A$ with $f \in k_B(K)$.

Assuming~(3), let us verify~(1).
Suppose that $I \subseteq B$ is a non-zero, closed ideal. 
Then $h_B(I) \subseteq \Pi_B$ is a closed and proper subset.
By assumption, we find a non-zero element $f \in A$ with $f \in k_B(h_B(I)) = I$.
Thus, $I \cap A \neq \{0\}$.
\end{proof}
	
We will have use for relating the ISP and the IIP for an inclusion $A \subseteq B$ of a dense $*$-subalgebra into a \ca{} to statements about $C^*$-seminorms on the $*$-algebra $A$. 
In order to state and prove our next result, let us introduce some notation. 
Denote by $R_A^B$ the collection of $*$-representations of $A$ which extend to $*$-representations of the \ca{} $B$. 
Assuming that $A$ admits a $C^*$-envelope $C^*(A)$ (then $A$ is called a $G^*$-algebra 
\cite[Definition~10.1.1]{Pal01BAlg2}), we may identify~$R_A^B$ with
\[
\big\{ \pi \in R_A^{C^*(A)} \mid \ker(\pi_\ast) \supseteq \ker (C^*(A) \to B) \big\}
\]
where $\pi_\ast$ is the natural extension of $\pi$ to $C^*(A)$, and $C^*(A) \to B$ is the canonical surjection. {Indeed, this identification is justified by observing that by the universal property of $C^*(A)$, $\pi \in R_A^B$ if and only if $\pi_\ast$ factors through the surjection $C^*(A) \to B$. }
	
\begin{prp}
\label{prp:norm-reformulations}
Let $A \subseteq B$ be an inclusion of a dense $*$-subalgebra $A$ into a \ca{} $B$.
For $\pi \in R_A^B$, denote by $\overline{\pi}$ the extension to $B$. 
The following statements hold:
\begin{enumerate}
\item 
The inclusion $A \subseteq B$ has the ISP if and only if for every $\pi, \varrho \in R_A^B$ the containment $\ker(\pi) \subseteq \ker(\varrho)$ in $A$ implies $\Vert\overline{\varrho}(b) \Vert \leq \Vert \overline{\pi}(b)\Vert$ for all $b \in B$.
\item 
The inclusion $A \subseteq B$ has the IIP if and only if for every $\pi \in R_A^B$, the statement $\ker(\pi) = \{0\}$ in $A$ implies $\Vert \overline{\pi}(b) \Vert = \Vert b\Vert_B$ for all $b \in B$. 
\end{enumerate}
\end{prp}
\begin{proof}
We begin by proving (1). 
We first note that $\pi, \varrho \in R_A^B$ satisfy $\ker(\overline{\pi}) \subseteq \ker(\overline{\rho})$ if and only if $\Vert\overline{\varrho}(b) \Vert \leq \Vert \overline{\pi}(b)\Vert$ for all $b \in B$.

Suppose now that $A \subseteq B$ has the ISP, and let $\pi, \varrho \in R_A^B$ satisfy $\ker(\pi) \subseteq \ker(\varrho)$ in $A$.
Then
\[
\ker(\overline{\pi}) \cap A
= \ker(\pi)
\subseteq \ker(\varrho)
= \ker(\overline{\varrho}) \cap A.
\]
By \autoref{prp:CharISPGeneral}, we get $\ker(\overline{\pi}) \subseteq \ker(\overline{\varrho})$, as desired.

Conversely, suppose that $\ker(\pi) \subseteq \ker(\varrho)$ in $A$ implies $\ker(\overline{\pi}) \subseteq \ker(\overline{\varrho})$ for all $\pi, \varrho \in R_A^B$.
To show that $A \subseteq B$ has the ISP, we verify statements~(4) in \autoref{prp:CharISPGeneral}.
Let $I,J \subseteq B$ be closed ideals satisfying $I \cap A \subseteq J \cap A$.
Choose $*$-representations~$\alpha$ and~$\beta$ of $B$ with $I=\ker(\alpha)$ and $J=\ker(\beta)$.
Set $\pi=\alpha|_A$ and $\varrho=\beta|_A$, which belong to $R_A^B$ and satisfy $\alpha=\overline{\pi}$ and $\beta=\overline{\varrho}$.
We have
\[
\ker(\pi)
= \ker(\overline{\pi}) \cap A
= I \cap A
\subseteq J \cap A
= \ker(\overline{\varrho}) \cap A
=\ker(\varrho).
\]
By assumption, we get $\ker(\overline{\pi}) \subseteq \ker(\overline{\varrho})$, and thus
\[
I = \ker(\overline{\pi}) \subseteq \ker(\overline{\varrho}) = J.
\]

Statement~(2) is shown similarly.
\end{proof}

\section{Permanence properties of the ISP and IIP for groups}
\label{sec:perm}

In this section we investigate permanence of the ideal separation property (ISP) and the ideal intersection property (IIP)  for locally compact groups when passing to quotients (\autoref{prp:Quotient}), inductive limits (\autoref{prp:DirectLimit}) and inverse limits (\autoref{prp:InvLim}).
We also show that a locally compact group has the ISP (the IIP) if sufficiently many of its closed subgroups do (\autoref{prp:ApproxByClosedSubgroups}), and in particular that the ISP and the IIP are local properties (\autoref{prp:LocalProps}).

\medskip

Every locally compact group will be assumed to come equipped with a fixed left Haar measure. 
Recall from \autoref{dfn:ISP-IIP-groups} that a locally compact group $G$ is said to have the ISP or the IIP if this property holds for the inclusion $L^1(G) \subseteq C^*_\red(G)$.

The next result generalizes results about the passage of $*$-regularity and $C^*$-uniqueness to quotient groups;
see \autoref{prp:StarRegQuotient}.

\begin{prp}
\label{prp:Quotient}
Let $G$ be a locally compact group, and suppose $N \subseteq G$ is a closed, normal subgroup. 
The following statements hold:
\begin{enumerate}
\item 
If $G$ has the ISP and $N$ is amenable, then $G/N$ has the ISP.
\item 
If $G$ has the IIP and $N$ is compact, then $G/N$ has the IIP. 
\end{enumerate}
\end{prp}
\begin{proof}
Using that $N$ is a closed, normal subgroup, we obtain a well-defined map $\varrho \colon L^1 (G) \to L^1 (G/N)$ given by
\begin{align*}
\varrho(f)(xN) = \int_N f(xn) dn
\end{align*}
for $f \in L^1(G)$ and $n \in N$.
By \cite[Proposition~3.4.5 and Theorem~3.5.4]{ReiSte00HarmAna}, $\varrho$ is a surjective $*$-homomorphism. 

Assuming that $N$ is also amenable (which holds under the assumption of~(1), as well as of~(2)), the map $\varrho$ extends to a surjective $*$-homomorphism $\pi \colon  C^*_\red(G) \to  C^*_\red(G/N)$.
(This is folklore but the details can be found in e.g. \cite[Corollary~2.4]{Man21ExactVsCExact}). 
We obtain the following commutative diagram:
\[
\xymatrix{
L^1( G ) \ar@{^{(}->}[r] \ar@{->>}[d]_{\varrho}
& C^*_\red(G) \ar@{->>}[d]^{\pi} \\
L^1(G/N) \ar@{^{(}->}[r]
& C^*_\red(G/N).
}
\]
We have
\[
\pi(\pi^{-1}(I) \cap L^1(G)) = I \cap L^1(G/N), \andSep
\pi^{-1}(I) \cap L^1(G) = \pi^{-1}( I \cap L^1(G/N) )
\]
for every closed ideal $I \subseteq C^*_\red(G/N)$.

\medskip

(1)
Assuming that $G$ has the ISP and that $N$ is amenable, we verify that~$G/N$ has the ISP.
Let $I,J \subseteq  C^*_\red(G/N)$ be closed ideals with $I \cap L^1(G/N) \subseteq J \cap L^1(G/N)$.
Then
\[
\pi^{-1}(I) \cap L^1(G) 
= \pi^{-1}( I \cap L^1(G/N) )
\subseteq \pi^{-1}( I \cap L^1(G/N) )
= \pi^{-1}(I) \cap L^1(G) 
\]
Applying that $G$ has the ISP, we get $\pi^{-1}(I) \subseteq \pi^{-1}(J)$, and thus $I \subseteq J$.

\medskip

(2)
Assuming that $G$ has the IIP and that $N$ is compact, we verify that~$G/N$ has the IIP.
We first identify $G^*_\red(G/N)$ with a direct summand of $C^*_\red(G)$.
For full group \ca{s} a similar result is contained in \cite[Lemma~5.2]{Lip72RepThyAlmCtdGps}.

The normalized Haar measure on $N$ defines a central idempotent $z$ in the measure algebra $M(G)$, which we identify with the multiplier algebra $M(L^1(G))$ of~$L^1(G)$.
We have $\ker(\varrho) = (1-z)L^1(G)$.
Using that $L^1(G)$ contains a contractive approximate identity for $C^*_\red(G)$, we obtain a natural contractive $*$-homomorphism $\varphi \colon M(L^1(G)) \to M(C^*_\red(G))$.
Then $\varphi(z)$ is a central projection in~$M(C^*_\red(G))$, and we obtain that $\ker(\pi) = (1-\varphi(z))C^*_\red(G)$.
This means that $\varrho$ and $\pi$ induce bijections 
\[
zL^1(G) \to L^1(G/N), \andSep
\varphi(z)C^*_\red(G) \to C^*_\red(G/N).
\]

To verify that $G/N$ has the IIP, let $I \subseteq  C^*_\red(G/N)$ be a non-zero, closed ideal, and set
\[
J := \varphi(z)\pi^{-1}(I).
\]
Then $J$ is a closed ideal in $C^*_\red(G)$ with $\pi(J)=I$.
Since $G$ has the IIP, it follows that $J \cap L^1(G)$ is non-zero.
Since $J \cap L^1(G)$ is contained in $zL^1(G)$, and since $\varrho$ is injective on $zL^1(G)$, it follows that 
\[
\{0\} 
\neq \varrho(J \cap L^1(G))
\subseteq I \cap L^1(G/N),
\]
as desired.
\end{proof}

The following result is (implicitly) contained in \cite{Boi82CtdGpsPolyDual} and \cite{Boi84GpAlgsUniqueCNorm}.
We note that it is not clear if $C^*$-uniqueness passes to quotients by closed, normal subgroups, even in the case of discrete groups;
see \cite[Remark~3.6]{LeuNg04PermPropCUniqueGps}.

\begin{cor}
\label{prp:StarRegQuotient}
Let $G$ be a locally compact group, and suppose $N \subseteq G$ is a closed, normal subgroup. 
The following statements hold:
\begin{enumerate}
\item 
If $G$ is $*$-regular, then $G/N$ is $*$-regular.
\item 
If $G$ is $C^*$-unique and $N$ is compact, then $G/N$ is $C^*$-unique. 
\end{enumerate}
\end{cor}
\begin{proof}
As noted in \autoref{prp:CharStarReg}, a locally compact group is $*$-regular ($C^*$-unique) if and only if it is amenable and has the ISP (the IIP).
Using that amenability passes to quotients by closed, normal subgroups (\cite[Proposition~1.13]{Pat88Amen}), the result follows from \autoref{prp:Quotient}.
\end{proof}

The next result shows that a locally compact group has the ISP (the IIP) if `sufficiently' many closed subgroups do.
Analogous results were shown for $*$-regularity in \cite[Theorem~1]{Boi82StarRegSomeSolvGps}, and for $C^*$-uniqueness (implicitly) in \cite[Proposition~2.1]{LeuNg04PermPropCUniqueGps}.

\begin{thm}
\label{prp:ApproxByClosedSubgroups}
Let $G$ be a locally compact group, and let $(H_i)_{i \in I}$ be a family of closed subgroups of~$G$ such that every open, relatively compact subset of~$G$ is contained in some $H_i$.
The following statements hold:
\begin{enumerate}
\item 
If every $H_i$ has the ISP, then $G$ has the ISP.
\item 
If every $H_i$ has the IIP, then $G$ has the IIP. 
\end{enumerate}
\end{thm}
\begin{proof}
Our proof is modeled after that of \cite[Theorem~1]{Boi82StarRegSomeSolvGps}.

We first prove (1). 
To verify that $G$ has the ISP, we use \autoref{prp:norm-reformulations} for the inclusion $L^1(G) \subseteq C^*_\red(G)$.
Let $\pi, \varrho \in R_{L^1(G)}^{C^*_r (G)}$ with $\ker(\pi) \subseteq \ker(\varrho)$ in $L^1(G)$.
Let $f \in C_c (G)$, and set $K = \{g \in G \mid f(g) \neq 0\}$. 
Then $K$ is open and relatively compact, so by assumption we find $i_0 \in I$ for which $\overline{K} \subseteq H_{i_0}$. 
Then $H_{i_{0}}$ is an open subgroup of $G$, and thus we may assume that the Haar measure on $H_{i_0}$ is given by the restriction of the Haar measure on $G$.
Then $L^1(H_{i_0})$ embeds isometrically into $L^1(G)$, and we can restrict $\pi$ and $\varrho$ to $*$-representations of $L^1(H_{i_0})$, which then satisfy $\ker(\pi\vert_{L^1(H_{i_0})}) \subseteq \ker(\varrho\vert_{L^1(H_{i_0})})$.
Using that $H_{i_{0}}$ has the ISP, we get {
\begin{align*}
\Vert \rho\vert_{L^1(H_{i_0})} (f \vert_{H_{i_0}}) \Vert 
\leq \Vert \pi \vert_{L^1(H_{i_0})} (f\vert_{H_{i_0}}) \Vert.
\end{align*}
Noting that $\pi\vert_{L^1(H_{i_0})} (f\vert_{H_{i_0}}) = \pi(f)$}, and $\rho\vert_{L^1(H_{i_0})}(f\vert_{H_{i_0}}) = \rho(f)$, we get $\Vert \rho(f) \Vert \leq \Vert \pi(f)\Vert$. 
Since $f$ was arbitrary, this holds for all $f \in C_c (G)$, and by density we deduce it holds for all $f \in C^*_r (G)$. 
By \autoref{prp:norm-reformulations}, we conclude that $L^1(G) \subseteq C^*_r (G)$ has the ISP. 

\medskip
		
The proof of (2) is very similar. 
Let $\pi \in R_{L^1(G)}^{C^*_r (G)}$ be such that $\ker(\pi) = \{0\}$ in $L^1(G)$. 
Let $f \in C_c (G)$ be arbitrary, and set $K = \{g \in G \mid f(g) \neq 0\}$. 
As above, there is $i_0 \in I$ for which $\overline{K} \subseteq H_{i_0}$, and using that $H_{i_{0}}$ has the IIP, we get
\begin{align*}
\Vert \pi\vert_{L^1(H_{i_0})} (f \vert_{H_{i_0}}) \Vert 
= \Vert f \Vert_{C^*_r (H_{i_0})}.
\end{align*}
Noting that $C^*_r (H_{i_0}) \hookrightarrow C^*_r (G)$ isometrically, we deduce that $\Vert \pi(f) \Vert = \Vert f\Vert_{C^*_r (G)}$. 
Since $f$ was arbitrary, this holds for all $f \in C_c (G)$, and by density we deduce it holds for all $f \in C^*_r (G)$. 
By \autoref{prp:norm-reformulations} we conclude that $L^1(G) \subseteq C^*_r (G)$ has the IIP. 
\end{proof}

Using \autoref{prp:ApproxByClosedSubgroups} we can easily establish that both the ISP and the IIP are local properties for groups. 
Analogous results for $*$-regularity and $C^*$-uniqueness were shown in \cite[Corollary~1]{Boi82StarRegSomeSolvGps} and \cite[Proposition~2.1]{LeuNg04PermPropCUniqueGps}, respectively.

\begin{cor}
\label{prp:LocalProps}
Let $G$ be a locally compact group.
The following statements hold:
\begin{enumerate}
\item 
If every open, compactly generated subgroup of $G$ has the ISP, then $G$ has the ISP.
\item 
If every open, compactly generated subgroup of $G$ has the IIP, then $G$ has the IIP.
\end{enumerate}
\end{cor}
\begin{proof}
Every open, relatively compact subset of $G$ is contained in some open, compactly generated subgroup of $G$. 
Therefore, the result follows from \autoref{prp:ApproxByClosedSubgroups}.
\end{proof}
	
Next, we consider preservation of the ISP and the IIP when passing to inductive limits.
We restrict to the case of discrete groups, which ensures that inductive limits exist.
(For a discussion of the subtleties in defining even a group topology on the inductive limit of non-discrete, locally compact groups, we refer to \cite{TatShiHir98IndLimTopGps}.)

If $\Gamma_1 \to \Gamma_2 \to \cdots$ is a sequential inductive system of discrete groups with injective connecting maps, then the natural maps $\Gamma_n \to \varinjlim_k \Gamma_k$ are automatically injective, which shows that the conditions on the kernels in the next result are automatically satisfied in this case.

\begin{thm}
\label{prp:DirectLimit}
Let $\Gamma = \varinjlim_\lambda \Gamma_\lambda$ be an inductive limit of discrete groups.
The following statements hold.
\begin{enumerate}
\item 
If each $\Gamma_\lambda$ has the ISP, and if the kernel of each $\Gamma_\lambda \to \Gamma$ is amenable, then~$\Gamma$ has the ISP.
\item 
If each $\Gamma_\lambda$ has the IIP, and if the kernel of each $\Gamma_\lambda \to \Gamma$ is finite, then~$\Gamma$ has the IIP.
\end{enumerate}
\end{thm}
\begin{proof}
For each $\lambda$, let $\Delta_\lambda$ denote the image of the natural map $\Gamma_\lambda \to \Gamma$.
Then $(\Delta_\lambda)_\lambda$ is a directed family of (automatically closed) subgroups of $\Gamma$.
Hence, every open and relatively compact (=finite) subset of $\Gamma$ is contained in $\Delta_\lambda$ for some $\lambda$, which shows that $(\Delta_\lambda)_\lambda$  satisfies the general assumptions of \autoref{prp:ApproxByClosedSubgroups}.

(1)
For each $\lambda$, using that $\Delta_\lambda$ is the quotient of $\Gamma_\lambda$ by an amenable subgroup, and that $\Gamma_\lambda$ has the ISP, it follows from \autoref{prp:Quotient} that $\Delta_\lambda$ has the ISP.
Consequently, $\Gamma$ has the ISP by \autoref{prp:ApproxByClosedSubgroups}.

(2)
Analogously to~(1), we see that each $\Delta_\lambda$ has the IIP by \autoref{prp:Quotient}, and thus $\Gamma$ has the IIP by \autoref{prp:ApproxByClosedSubgroups}.
\end{proof}

We deduce that $*$-regularity passes to inductive limits of discrete groups, which seems to not have been noticed so far.

\begin{cor}
Let $\Gamma = \varinjlim_\lambda \Gamma_\lambda$ be an inductive limit of discrete groups.
The following statements hold:
\begin{enumerate}
\item 
If each $\Gamma_\lambda$ is $*$-regular, then $\Gamma$ is $*$-regular.
\item
If each $\Gamma_\lambda$ is $C^*$-unique and if all connecting maps $\Gamma_\lambda \to \Gamma_\kappa$ are injective, then $\Gamma$ is $C^*$-unique.
\end{enumerate}
\end{cor}
\begin{proof}
(1)
Since $*$-regular groups are amenable, and since amenability passes to (closed) subgroups, we see that every subgroup of a discrete, $*$-regular group is amenable.
Therefore, the kernel of the map $\Gamma_\lambda \to \Gamma$ is automatically amenable for every $\lambda$, and the result follows from \autoref{prp:DirectLimit}.

(2)
The assumption implies that $\Gamma_\lambda \to \Gamma$ is injective for every $\lambda$, and the result follows from \autoref{prp:DirectLimit}.
\end{proof}

\begin{rmk}
In light of $*$-regularity passing to inductive limits of discrete groups, it would be of great interest to determine if $*$-regularity passes to extensions of discrete groups.
A positive answer (even just for extensions by finite groups and~$\ZZ$) would immediately provide a good partial answer to \autoref{qst:discrete-amenable-is-*-regular} by showing that elementary amenable groups are $*$-regular.
Indeed, a result of Osin \cite[Corollary~2.1]{Osi02ElemClassesGps} shows that the class of elementary amenable groups is the smallest class~$\mathcal{C}$ of groups which contains the trivial group and is closed under inductive limits, and such that $\Gamma \in \mathcal{C}$ for every extension
\[
1 \to N \to \Gamma \to \Gamma/N \to 1
\]
with $N \in \mathcal{C}$ and $\Gamma/N$ a finite group or $\Gamma/N = \ZZ$.

Since in the result of Osin it suffices to assume that the class is closed under inductive limits with injective connecting maps, one could similarly deduce that elementary amenable groups are $C^*$-unique if one can show that $C^*$-uniqueness passes to extensions by finite groups and $\ZZ$.
\end{rmk}

We deduce from \autoref{prp:ApproxByClosedSubgroups} the existence of certain maximal subgroups that have the ISP or the IIP.
An analogous result for maximal $*$-regular subgroups of discrete groups was shown in \cite[Proposition~3.2]{LeuNg04PermPropCUniqueGps}.

\begin{thm}
Let $G$ be a locally compact group.
The following statements hold:
\begin{enumerate}
\item 
Every (normal) open subgroup $H \leq G$ having the ISP is contained in a maximal (normal) open subgroup of $G$ having the ISP.
\item 
Every (normal) open subgroup $H \leq G$ having the IIP is contained in a maximal (normal) open subgroup of $G$ having the IIP.
\end{enumerate}
\end{thm}
\begin{proof}
To verify (1), let $H \leq G$ be an open (normal) subgroup that has the ISP. 
Let $\mathcal{L}$ denote the set of open (normal) subgroups of $G$ that contain~$H$ and that have the ISP.
We order $\mathcal{L}$ by inclusion.
The existence of maximal elements in $\mathcal{L}$ follows once we have verified the assumptions of Zorn's lemma for $\mathcal{L}$.

Given a chain $\mathcal{C}$ in $\mathcal{L}$, we need to find an upper bound of $\mathcal{C}$ in $\mathcal{L}$.
Set $H' := \bigcup_{C \in \mathcal{C}} C$, which is an open (normal) subgroup of $G$ containing $H$.
Since $\mathcal{C}$ is a chain, every relatively compact subset of $H'$ is contained in some $C \in \mathcal{C}$.
This shows that the family $\mathcal{C}$ of open (hence also closed) subgroups of $H'$ satisfies the assumptions of \autoref{prp:ApproxByClosedSubgroups}, and since each group in $\mathcal{C}$ has the ISP, we deduce that $H'$ has the ISP.
Hence $H'$ belongs to $\mathcal{L}$, and it is clearly an upper bound for $\mathcal{C}$.

An analogous argument proves~(2).
\end{proof}

We now turn to inverse limits.
As for inductive limits, inverse limits of locally compact groups only exist under additional assumptions.
We restrict our attention to the important case of an inverse limit $G = \varprojlim_\lambda G_\lambda$ where each map $G \to G_\lambda$ is surjective and has compact kernel $K_\lambda$.
Then $(K_\lambda)_\lambda$ is a downward directed family of compact, normal subgroups with $\bigcap_\lambda K_\lambda = \{1\}$, and given such a family of compact, normal subgroups we can realize $G$ as an inverse limit of the quotients $G/K_\lambda$.
In this setting, we show that $G$ has the ISP (the IIP) if and only if each~$G_\lambda$ has the ISP (the IIP).
We discuss in \autoref{rmk:InvLim} how our result generalizes similar statements about $*$-regularity and $C^*$-uniqueness from \cite[Lemma~9]{Boi82CtdGpsPolyDual} and \cite[Corollary~2.3]{LeuNg04PermPropCUniqueGps}.

\begin{thm}
\label{prp:InvLim}
Let $G$ be a locally compact group which can be realized as an inverse limit $G = \varprojlim_\lambda G_\lambda$ such that each canonical map $G \to G_\lambda$ is surjective and has compact kernel. 
The following statements hold:
\begin{enumerate}
\item 
The group $G$ has the ISP if and only if every $G_\lambda$ has the ISP.
\item 
The group $G$ has the IIP if and only if every $G_\lambda$ has the IIP.
\end{enumerate}
\end{thm}
\begin{proof}
For each $\lambda$, let $K_\lambda$ denote the kernel of the natural map $G \to G_\lambda$, so that $G_\lambda \cong G/K_\lambda$. 
By \autoref{prp:Quotient}, the ISP and the IIP pass to quotients by compact, normal subgroups, which shows the forward implications in~(1) and~(2).

Given $\lambda$, using that $K_\lambda$ is a compact, normal subgroup of $G$, we have seen in the proof of \autoref{prp:Quotient} that there exist natural surjective $*$-homomorphisms $\pi_\lambda \colon C^*_\red(G) \to C^*_\red(G_\lambda)$ and $\varrho_\lambda \colon L^1(G) \to L^1(G_\lambda)$ such that the following diagram commutes:
\[
\xymatrix{
L^1( G ) \ar@{^{(}->}[r] \ar@{->>}[d]_{\varrho_\lambda}
& C^*_\red(G) \ar@{->>}[d]^{\pi_\lambda} \\
L^1(G_\lambda) \ar@{^{(}->}[r]
& C^*_\red(G_\lambda).
}
\]

Set $I_\lambda := \ker(\pi_\lambda)$, which is a closed ideal in $C^*_\red(G)$.

Claim~1: \emph{We have $\bigcap_\lambda I_\lambda = \{0\}$.}
We have seen in the proof of \autoref{prp:Quotient} that for each $\lambda$ there exists a central, contractive idempotent $z_\lambda \in M(G)$ such that {
\begin{equation}\label{eq:kernel-identifications}
	\ker(\varrho_\lambda) = (1-z_\lambda)L^1(G), \andSep
	I_\lambda = \ker(\pi_\lambda) = (1-\varphi(z_\lambda))C^*_\red(G),
\end{equation}}
where $\varphi \colon M(G) \to M(C^*_\red(G))$ denotes the natural, contractive $*$-homomorphism.
Using that $\bigcap_\lambda K_\lambda = \{1\}$ {and \cite[Proposition 3.5.6]{ReiSte00HarmAna}}, it follows that $\lim_\lambda \|z_\lambda f - f\|_1 = 0$ for every $f \in L^1(G)$. 
Given $a \in C^*_\red(G)$, let us verify that $\lim_\lambda \|\varphi(z_\lambda) a - a\| = 0$.
Given $\varepsilon>0$, choose $f \in L^1(G)$ with $\| a - f \| < \varepsilon$.
Then
\begin{align*}
\| \varphi(z_\lambda)a - a \|
&\leq \| \varphi(z_\lambda)a - \varphi(z_\lambda)f \| + \| \varphi(z_\lambda)f - f \| + \| f -a \| \\
&\leq \| \varphi(z_\lambda) \| \|a - f \| + \| z_\lambda f - f \|_1 + \| f -a \|
< 3 \varepsilon.
\end{align*}

Now, given $a \in \bigcap_\lambda I_\lambda \subseteq C^*_\red(G)$, we get {$(1-z_\lambda)a=0$} 
and thus $z_\lambda a = a$ for each $\lambda$, which then gives $a=0$.
Thus, $\bigcap_\lambda I_\lambda = \{0\}$.
This proves the claim.

\medskip

Claim~2: \emph{Given a closed ideal $I \subseteq C^*_\red(G)$, we have
\[
\pi_\lambda(I) \cap L^1(G_\lambda) = \pi_\lambda(I \cap L^1(G))
\]
for every $\lambda$.}
Indeed, the inclusion `$\supseteq$' is clear.
To show the converse inclusion, let $a \in \pi_\lambda(I) \cap L^1(G_\lambda)$, and choose $b \in I$ and $f \in L^1(G)$ such that
\[
\pi_\lambda(b) = a, \andSep
\varrho_\lambda(f) = a.
\]
Then $z_\lambda f \in L^1(G)$.
Since $I$ is a \ca{}, every element factors and we can choose $b_1,b_2 \in I$ such that $b=b_1b_2$.
Then
\[
\varphi(z_\lambda)b
= \varphi(z_\lambda)b_1b_2
\in C^*_\red(G)I
\subseteq I.
\]
Note that
{ using \autoref{eq:kernel-identifications} we obtain
\[
\pi_\lambda(z_\lambda f)
=\pi_\lambda(f)
=a
=\pi_\lambda(b)
=\pi_\lambda(\varphi(z_\lambda)b)
\]} 
and since $\pi_\lambda$ is injective on {$\varphi(z_\lambda)C^*_\red(G)$ by \autoref{eq:kernel-identifications},}  
we get $z_\lambda f= \varphi(z_\lambda) b \in I \cap L^1(G)$.
This proves the claim.

\medskip

(1)
Assuming that each $G_\lambda$ has the ISP, let us verify that $G$ has the ISP.
Let $I,J \subseteq C^*_\red(G)$ be closed ideals with $I \cap L^1(G) \subseteq J \cap L^1(G)$.
For each $\lambda$, using Claim~2, we deduce that
\[
\pi_\lambda(I) \cap L^1(G_\lambda)
= \pi_\lambda(I \cap L^1(G))
\subseteq \pi_\lambda(J \cap L^1(G))
= \pi_\lambda(J) \cap L^1(G_\lambda).
\]
Using that $G_\lambda$ has the ISP, we get $\pi_\lambda(I) \subseteq \pi_\lambda(J)$, and thus $I \subseteq J + I_\lambda$.
Applying Claim~1 at the last step, we obtain
\[
I 
\subseteq \bigcap_\lambda ( J + I_\lambda ) 
= J + \bigcap_\lambda I_\lambda
= J,
\]
as desired. 

\medskip

(2)
Assuming that each $G_\lambda$ has the IIP, let us verify that $G$ has the IIP.
Let $I \subseteq C^*_\red(G)$ be a non-zero, closed ideal.
By Claim~1, we have $\bigcap_\lambda I_\lambda = \{0\}$, and we find $\lambda$ such that $I \nsubseteq I_\lambda$.
Then $\pi_\lambda(I) \subseteq C^*_\red(G_\lambda)$ is a non-zero, closed ideal.
Using that $G_\lambda$ has the IIP at the first step, and using Claim~2 at the second step, we get
\[
\{0\} 
\neq \pi_\lambda(I) \cap L^1(G_\lambda)
= \pi_\lambda(I \cap L^1(G) ),
\]
which shows that $I \cap L^1(G)$ is non-zero, as desired.
\end{proof}

\begin{rmk}
\label{rmk:InvLim}
Let $G$ be a locally compact group, and let $(K_\lambda)_\lambda$ be a downward directed family of compact, normal subgroups with $\bigcap_\lambda K_\lambda = \{1\}$.
By \cite[Lemma~9]{Boi82CtdGpsPolyDual}, the group $G$ is $*$-regular if and only if all $G/K_\lambda$ are $*$-regular.
Similarly, by \cite[Corollary~2.3]{LeuNg04PermPropCUniqueGps}, the group $G$ is $C^*$-unique if and only if $G/K_\lambda$ is $C^*$-unique for large enough $\lambda$.

\autoref{prp:InvLim} recovers and generalizes these results (and even strengthens the result about $C^*$-uniqueness).
Indeed, by \cite[Proposition~1.13]{Pat88Amen}, if $N$ is a closed, normal subgroup in $G$, then $G$ is amenable if and only if $N$ and $G/N$ are amenable.
Since compact groups are amenable, it follows that $G$ is amenable if and only if $G/K_\lambda$ is amenable for some $\lambda$ (equivalently, for every $\lambda$).
Therefore, \cite[Lemma~9]{Boi82CtdGpsPolyDual} follows from \autoref{prp:InvLim} using that by \autoref{prp:CharStarReg} a locally compact group is $*$-regular if and only if it is amenable and satisfies the ISP.
Similarly, we recover \cite[Corollary~2.3]{LeuNg04PermPropCUniqueGps}, and in fact we obtain that $G$ is $C^*$-unique if and only if all $G/K_\lambda$ are $C^*$-unique.
\end{rmk}

Given a locally compact group $G$, let $G_0$ denote the connected component of the identity, which is always a closed, normal subgroup.
One says that $G$ is \emph{connected} if $G=G_0$, that $G$ is \emph{almost connected} if $G/G_0$ is compact, and that $G$ is \emph{totally disconnected} if $G_0=\{1\}$.
The quotient group $G/G_0$ is totally disconnected, and it follows from van Dantzig's theorem (see, for example \cite[Theorem~1.6.7]{Tao14Hilb5thPbm}) that there exists an open, compact subgroup $U \subseteq G/G_0$.
Then the pre-image of $U$ in $G$ is an open subgroup $V \subseteq G$ such that $V/V_0$ ($=V/G_0$) is compact.
This shows that every locally compact group contains an almost connected, open subgroup.

By the Gleason-Yamabe theorem (\cite[Theorem~4.6]{MonZip55TopTransfGps}), every almost connected, locally compact group is an inverse limit of (finite-dimensional, real) Lie groups.
(See also \cite[Theorem~1.6.1]{Tao14Hilb5thPbm}.)
As a consequence of \autoref{prp:InvLim}, we see that the ISP and the IIP for an almost connected, locally compact group is determined by its Lie group quotients.
This generalizes a similar statement for $C^*$-uniqueness from \cite[Proposition~2.2(b)]{LeuNg04PermPropCUniqueGps}.
One also obtains that an almost connected, locally compact group is $*$-regular if and only if all of its Lie group quotients are $*$-regular.

\begin{cor}
\label{prp:LieQuotients}
Let $G$ be an almost connected, locally compact group, and set
\[
\mathcal{L} := \big\{ K \subseteq G : K \text{ is a compact, normal subgroup such that $G/K$ is a Lie group} \big\}.
\]
Then $G$ has the ISP (the IIP) if and only if $G/K$ has the ISP (the IIP) for every $K \in \mathcal{L}$.
\end{cor}
\begin{proof}
This follows from \autoref{prp:InvLim} using that the family $\mathcal{L}$ is downward directed (possibly with a minimal element $\{1\}$ if $G$ is already a Lie group).
\end{proof}

\section{Ideal separation in tensor products of \texorpdfstring{$C^*$}{C*}-algebras}
\label{sec:TensProd}

Given \ca{s} $A$ and $B$, we observe that the inclusion of the algebraic tensor product $A \odot B$ in the minimal tensor product $A \tensMin B$ always has the ideal intersection property (IIP);
see \autoref{prp:AlgTP}.
For the ideal separation property (ISP) this is no longer true:
The inclusion $A \odot B \subseteq A \tensMin B$ has the ISP if and only if the pair of \ca{s} $A$ and~$B$ has Tomiyama's property~(F);
see \autoref{prp:CharPropF}.
Analogous results holds for the inclusion of the projective tensor product $A \tensProj B$ in $A \tensMin B$.

Using this, we show that for two locally compact groups $G$ and $H$, the Cartesian product $G \times H$ has the IIP if and only if $G$ and $H$ have the IIP.
Further, $G \times H$ has the ISP if and only if $G$ and $H$ have the ISP and the pair of reduced group \ca{s} $C^*_\red(G_1)$ and $C^*_\red(G_2)$ has Tomiyama's property~(F);
see \autoref{prp:ProductGp}.
Since Tomiyama's property~(F) is automatic whenever one of the \ca{s} is exact, we deduce that for two locally compact groups $G$ and $H$, at least one of which is exact, $G \times H$ has the ISP if and only if $G$ and $H$ do;
see \autoref{prp:ISP-PrductExactGp}.

\medskip

We begin by recalling some basic facts about algebraic, injective, projective, minimal and maximal tensor products of \ca{s}.
Let $A$ and $B$ be two \ca{s}.
We use $A \odot B$ to denote the algebraic tensor product as $\CC$-vector spaces, which is a complex $\ast$-algebra with involution induced by $(a \otimes b)^* := a^* \otimes b^*$.
		
Viewing $A$ and~$B$ as Banach spaces, we let $\|\cdot\|_\varepsilon$ and $\|\cdot\|_\pi$ denote the injective and projective norms on $A \odot B$, with the respective completions denoted by $A \tensInj B$ and $A \tensProj B$;
see \cite[Sections~2 and~3]{Rya02IntroTensProdBSp} for details.
We further let $\|\cdot\|_{\mathrm{min}}$ and $\|\cdot\|_{\mathrm{max}}$ denote the minimal and maximal $C^*$-norms, with completions denoted by $A \tensMax B$ and $A \tensMin B$;
see \cite[Section~II.9]{Bla06OpAlgs} for details.
		
While $A \otimes B$ and $A \tensMax B$ are \ca{s}, $A \tensProj B$ is `only' a Banach $\ast$-algebra.
Even worse, $A \tensInj B$ is in general not even a Banach algebra, since the multiplication on $A \odot B$ may not extend to this completion;
see \cite[Corollary~4]{Ble88GeomTensProdCAlg}.
We have
\[
\|\cdot\|_\varepsilon
\leq \|\cdot\|_{\mathrm{min}}
\leq \|\cdot\|_{\mathrm{max}}
\leq \|\cdot\|_\pi,
\]
and in general all norms can be different.
We obtain contractive maps with dense ranges as shown in the following commutative diagram:
\[
\xymatrix{
A \tensMax B \ar@{->>}[r]
& A \tensMin B \ar[d] \\
A \tensProj B \ar[u] \ar@{^{(}->}[r]
& A \tensInj B.
}
\]
The map from the maximal to the minimal tensor product is a quotient map, and it is an isomorphism whenever $A$ or $B$ are nuclear (and also sometimes when neither~$A$ nor~$B$ are nuclear \cite{Pis20NonNuclWEP-LLP}).
Haagerup showed in \cite[Proposition~2.2]{Haa85GrothendieckIneq} that the canonical map $A \tensProj B \to A \tensInj B$ is injective.
It follows that the map $A \tensProj B \to A \otimes B$ is injective as well {since the first map in a composition of maps making up an injection must be injective}.
We will therefore usually identify $A \odot B$ and $A \tensProj B$ with subalgebras of $A \otimes B$:
\[
A \odot B 
\subseteq A \tensProj B 
\subseteq A \tensMin B.
\]
		
Given $\ast$-homomorphisms $\pi_1 \colon A_1 \to A$ and $\pi_2 \colon B_1 \to B$, the naturally induced $\ast$-homomorphism $A_1 \odot B_1 \to A \odot B$ extends to the respective completions, when both tensor products are equipped with the minimal (maximal, projective) norm.
If $\pi_1$ and $\pi_2$ are injective, then so are the maps $A_1 \odot B_1 \to A \odot B$ and $A_1 \tensMin B_1 \to A \tensMin B$, but $A_1 \tensMax B_1 \to A \tensMax B$ may not be injective.
If $\pi_1$ and $\pi_2$ are surjective, then so are  $A_1 \odot B_1 \to A \odot B$, as well as $A_1 \tensMin B_1 \to A \tensMin B$ and $A_1 \tensMax B_1 \to A \tensMax B$.
		
If $J$ is a closed ideal in $B$, then the kernel of the map $A \odot B \to A \odot (B/J)$ is equal to $A \odot J$.
Thus, the sequence
\[
0 \to A \odot J \to A \odot B \to A \odot (B/J) \to 0
\]
is exact.
The maximal tensor product is also exact in this sense:
The sequence 
\[
0 \to A \tensMax J \to A \tensMax B \to A \tensMax (B/J) \to 0
\]
is exact.
This is no longer true for the minimal tensor product, and one says that $A$ is \emph{exact} if taking minimal tensor products with $A$ preserves short exact sequences.
		
The projective tensor product of \ca{s} has the remarkable property of preserving exact sequences and isometric embeddings induced by sub-\ca{s}.
First, analogous to the minimal tensor product, if $A_1 \subseteq A$ and $B_1 \subseteq B$ are sub-\ca{s}, then the map $A_1 \odot B_1 \to A \odot B$ is isometric for the projective tensor products and induces a natural isometric embedding $A_1 \tensProj B_1 \to A \tensProj B$;
see \cite[Theorem~1]{GupJai20ProjTensProjCAlg}.
Second, analogous to the maximal tensor product, if $J \subseteq B$ is a closed ideal, then the sequence
\[
0 \to A \tensProj J \to A \tensProj B \to A \tensProj (B/J) \to 0
\]
is exact, that is, the map $A \tensProj J \to A \tensProj B$ is an isometric embedding and its image agrees with the kernel of the quotient map $A \tensProj B \to A \tensProj (B/J)$;
see \cite[Proposition~2.5]{Rya02IntroTensProdBSp}, 
see also \cite[Lemma~1]{GupJai20ProjTensProjCAlg}.

\medskip

Next, we study closed, prime ideals in projective tensor products of \ca{s}.
By an \emph{ideal} we mean a ring-theoretic ideal, that is, an additive subgroup that is invariant under taking products on the left or right with elements from the containing algebra.
We note that ideals in \ca{s} are not necessarily subspaces;
see \cite[Examples~II.5.2.1]{Bla06OpAlgs}.

We begin by generalizing the well-known result that the minimal tensor product of prime \ca{s} is prime to the setting of projective tensor products.

\begin{lma}
\label{prp:TensProjPrime}
Let $A$ and $B$ be prime \ca{s}.
Then $A \tensProj B$ is prime.
\end{lma}
\begin{proof}
The proof is based on ideas in the proof of \cite[Lemma~2.13]{BlaKir04PureInf}.
To reach a contradiction, let $I_1, I_2 \subseteq A \tensProj B$ be nonzero ideals such that $I_1I_2 = \{0\}$.
Let $\overline{I_1}$ and $\overline{I_2}$ denote the closures in~$A \tensProj B$.
By \cite[Corollary~2]{GupJai20ProjTensProjCAlg}, there exist nonzero elementary tensors $a_1 \otimes b_1 \in \overline{I_1}$ and $a_2 \otimes b_2 \in \overline{I_2}$.
Since $A$ is prime and $a_1,a_2 \neq 0$, we obtain $x \in A$ such that $a_1xa_2 \neq 0$.
Similarly, we get $y \in B$ such that $b_1yb_2 \neq 0$.
Then
\[
0 \neq (a_1xa_2) \otimes (b_1yb_2)
= (a_1 \otimes b_1) (x \otimes y) (a_2 \otimes b_2)
\in \overline{I_1}\cdot\overline{I_2}
\]
But $I_1I_2 = \{0\}$ implies $\overline{I_1} \cdot \overline{I_2} = \{0\}$, a contradiction.
\end{proof}
	
The next result establishes a natural bijection between the set of closed, prime ideals in~$A \tensProj B$, and the Cartesian product of the spaces of closed, prime ideals in~$A$ and in~$B$.
The result is probably well-known to experts, but we could not locate it in the literature.
The analogous statement identifying closed, maximal ideals (and maximal modular ideals) in~$A \tensProj B$ with Cartesian products of respective ideals in~$A$ and~$B$ is \cite[Theorem~9]{GupJai20ProjTensProjCAlg}.
An analogous statement for closed, prime ideals in the operator space projective tensor product is \cite[Theorem~6]{JaiKum11IdlsOpSpProjTPCAlg}.
	
\begin{prp}
\label{prp:PrimeIdlTensProj}
Let $A$ and $B$ be \ca{s}, and let $I \subseteq A \tensProj B$  be a closed, prime ideal.
Set
\[
J := \big\{ a \in A : a \odot B \subseteq I \big\}, \andSep
K := \big\{ b \in B : A \odot b \subseteq I \big\}.
\]
Then $J \subseteq A$ and $K \subseteq B$ are closed, prime ideals, and we have
\[
I = (J \tensProj B) + (A \tensProj K)
\subseteq A \tensProj B.
\]
		
Conversely, if $J \subseteq A$ and $K \subseteq B$ are closed, prime ideals, then $(J \tensProj B) + (A \tensProj K)$ is a closed, prime ideal in $A \tensProj B$.
\end{prp}
\begin{proof}
The proof is based on ideas in the proof of \cite[Lemma~2.13]{BlaKir04PureInf}.
First, it is straightforward to check that $J$ is a closed, additive subgroup.
To show that~$J$ is a left ideal, let $a \in J$ and $a_1 \in A$.
Given $b \in B$, we need to verify $(a_1a) \otimes b \in I$.
Since every element in a \ca{} can be written as a product of two elements (\cite[Proposition~II.3.2.1]{Bla06OpAlgs}), we find $b_1,b_2 \in B$ such that $b = b_1b_2$.
Then
\[
(a_1a) \otimes b 
= (a_1a) \otimes (b_1b_2)
= (a_1 \otimes b_1)(a \otimes b_2)
\in I.
\]
Similarly, one shows that $J$ is a right ideal.
To show that $J$ is a prime ideal, let $J_1,J_1 \subseteq A$ be ideals such that $J_1J_2 \subseteq J$.
Let $\overline{J_1}$ and $\overline{J_2}$ denote the closures in~$A$.
Since $J$ is closed, we have $\overline{J_1}\overline{J_2} \subseteq J$.
Then $\overline{J_1} \tensProj B$ and $\overline{J_2} \tensProj B$ are ideals in $A \tensProj B$ such that
\[
\big( \overline{J_1} \tensProj B \big)\big( \overline{J_2} \tensProj B \big)
\subseteq \big( \overline{J_1}\overline{J_2} \big) \tensProj B
\subseteq J \tensProj B
\subseteq I.
\]
Since $I$ is prime, we get $\overline{J_1} \tensProj B \subseteq I$ or $\overline{J_2} \tensProj B \subseteq I$, which entails $J_1 \subseteq \overline{J_1} \subseteq J$ or $J_2 \subseteq \overline{J_2} \subseteq J$.
Analogously, it follows that $K$ is a closed, prime ideal in~$B$.
		
It is clear that $I$ contains $(J \tensProj B) + (A \tensProj K)$.
To show the other inclusion, we consider the natural map $\pi \colon A \tensProj B \to (A/J) \tensProj (B/K)$.
Since the projective tensor product is exact (see \cite[Proposition~2.5]{Rya02IntroTensProdBSp}), the kernel of $\pi$ is $(J \tensProj B) + (A \tensProj K)$.
We will show that $\pi(I)=\{0\}$.

Since $\pi$ is a quotient map, $\pi(I)$ is a closed ideal in $(A/J) \tensProj (B/K)$.
To reach a contradiction, assume that $\pi(I) \neq \{0\}$.
Then $\pi(I)$ contains a nonzero, elementary tensor, say $x \otimes y \in \pi(I)$ for $x \in A/J$ and $y \in B/K$;
see \cite[Corollary~2]{GupJai20ProjTensProjCAlg}.
Choosing lifts $a \in A$ and $b \in B$, and using that the kernel of $\pi$ is contained in $I$, we deduce that $a \otimes b \in I$.
Let $J_1 \subseteq A$ and $K_1 \subseteq B$ denote the closed ideals generated by $a$ and $b$, respectively.
Then
\[
\big( J_1 \tensProj B \big) \big( A \tensProj K_1 \big)
\subseteq J_1 \tensProj K_1
\subseteq I.
\]
Using that $I$ is prime, it follows that $J_1 \tensProj B \subseteq I$ or $A \tensProj K_2 \subseteq I$, and therefore $a \in J$ or $b \in K$, which implies that $\pi(a \otimes b)=0$, a contradiction.
{By the above identification of the kernel of $\pi$ with $(J \tensProj B) + (A \tensProj K)$ we therefore have}
\[
I \subseteq \ker(\pi) = (J \tensProj B) + (A \tensProj K).
\]
		
Next, assume that $J \subseteq A$ and $K \subseteq B$ are closed, prime ideals.
As above, we consider the map $\pi \colon A \tensProj B \to (A/J) \tensProj (B/K)$, whose kernel is $(J \tensProj B) + (A \tensProj K)$.
By \autoref{prp:TensProjPrime}, the algebra $(A/J) \tensProj (B/K)$ is prime, and therefore the kernel of $\pi$ is a closed, prime ideal.
\end{proof}

An ideal $I$ in a ring $R$ is said to be semiprime if for every ideal $J$ in $R$ one has $J \subseteq I$ whenever $J^2 \subseteq I$.
It is known that an ideal is semiprime if and only if it is the intersection of some family of prime ideals;
see \cite[Section~10]{Lam01FirstCourse2ed}.
In \cite[Corollary~5.5]{GarKitThi23arX:SemiprimeIdls} it is shown that semiprime (not necessarily closed) ideals in \ca{s} are self-adjoint subspaces.
With view to the next result, it is therefore natural to ask if semiprime (not necessarily closed) ideals in projective tensor products of \ca{s} are self-adjoint subspaces.
	
\begin{cor}
\label{prp:SemiprimeIdlTensProj}
Let $A$ and $B$ be \ca{s}.
Then every closed, semiprime ideal in $A \tensProj B$ is a self-adjoint subspace.
\end{cor}
\begin{proof}
It suffices to show the result for closed, prime ideals, since being a self-adjoint subspace passes to intersections.
So let $I \subseteq A \tensProj B$ be a closed, prime ideal.
By \autoref{prp:PrimeIdlTensProj}, there are closed, prime ideals $J \subseteq A$ and $K \subseteq B$ such that
\[
I = (J \tensProj B) + (A \tensProj K).
\]
Since closed ideals in \ca{s} are self-adjoint subspaces, it follows that $I$ has the claimed properties.
\end{proof}
	
The next result is mostly contained in \cite[Lemma~2.13]{BlaKir04PureInf}.
	
\begin{prp}
\label{prp:PrimeIdlTensMin}
Let $A$ and $B$ be \ca{s}, and let $I \subseteq A \otimes B$  be a closed, prime ideal.
Set
\[
J := \big\{ a \in A : a \odot B \subseteq I \big\}, \andSep
K := \big\{ b \in B : A \odot b \subseteq I \big\}.
\]
Then $J \subseteq A$ and $K \subseteq B$ are closed, prime ideals, and we have
\[
(J \otimes B) + (A \otimes K)
\subseteq I
\subseteq I_{J,K} := \ker\big( A \otimes B \to (A/J)\otimes(B/K) \big).
\]
Moreover, 
\[
\big( J \otimes B + A \otimes K \big) \cap \big( A \odot B \big)
= I \cap A \odot B
= I_{J,K} \cap A \odot B 
= J \odot B + A \odot K
\]
and
\[
\big( J \otimes B + A \otimes K \big) \cap \big( A \tensProj B \big)
= I \cap A \tensProj B
= I_{J,K} \cap A \tensProj B 
= J \tensProj B + A \tensProj K.
\]

If $I = \overline{ I \cap A \tensProj B}$ (in particular, if $I = \overline{ I \cap A \odot B}$), then
\[
J \otimes B + A \otimes K 
= I.
\]
\end{prp}
\begin{proof}
As in the proof of \autoref{prp:PrimeIdlTensProj}, we see that $J$ and $K$ are closed, prime ideals, and that $(J \otimes B) + (A \otimes K) \subseteq I \subseteq I_{J,K}$.
(This is also contained in \cite[Lemma~2.13]{BlaKir04PureInf}.)
		
Using that the algebraic tensor product is exact, we see that the kernel of the map $A \odot B \to (A/J) \odot (B/K)$ is $(J \odot B) + (A \odot K)$, which implies that
\[
I_{J,K} \cap (A \odot B)
= (J \odot B) + (A \odot K).
\]
The inclusions
\[
(J \odot B) + (A \odot K)
\subseteq \big( (J \otimes B) + (A \otimes K) \big) \cap \big( A \odot B \big)
\subseteq I \cap (A \odot B)
\subseteq I_{J,K} \cap (A \odot B)
\]
are clear, and therefore are all equalities.
The equalities for the intersection with $A \tensProj B$ follow analogously, using that the projective tensor product is exact.
		
Lastly, let us assume that $I = \overline{ I \cap A \tensProj B}$.
Using at the last step that the sum of closed ideals in a \ca{} is closed, we get
\[
I
= \overline{ I \cap A \tensProj B }
= \overline{ J \tensProj B + A \tensProj K}
= \overline{ J \tensMin B + A \tensMin K}
= J \otimes B + A \otimes K,
\]
as desired.
\end{proof}

A pair of \ca{s} $A$ and $B$ is said to satisfy \emph{Tomiyama's property (F)} if closed ideals in $A \otimes B$ are separated by states of the form $\varphi\otimes\psi$, where $\varphi$ is a pure state on $A$ and $\psi$ is a pure state on $B$, \cite{Tom67FubiniTensProdCAlgs}.
We say that an ideal in the minimal tensor product $A \otimes B$ of two \ca{s} is a product ideal if it is of the form $I \otimes J$ for an ideal $I \subseteq A$ and an ideal $J \subseteq B$.

The next result characterizes Tomiyama's property (F) in terms of the ISP for the inclusion of the algebraic and the projective tensor product in the minimal tensor product of \ca{s}.
For the algebraic tensor product this is essentially contained in \cite[Proposition~2.16]{BlaKir04PureInf}, see also 
\cite[Proposition~5.1]{Laz10IdlsMinTensProdCAlgs}, and our contribution is the characterization in terms of the projective tensor product.
This is crucial for showing that if the Cartesian product $G \times H$ of two locally compact groups has the ISP, then the pair of \ca{s} $C^*_\red(G)$ and $C^*_\red(H)$ has Tomiyama's property~(F);
see \autoref{prp:ProductGp}.
	
\begin{thm}
\label{prp:CharPropF}
Let $A$ and $B$ be \ca{s}.
Then the following are equivalent:
\begin{enumerate}
\item
The pair ($A$, $B$) satisfies Tomiyama's property (F).
\item
Every closed ideal in $A \tensMin B$ is the closure of the sum of all contained product ideals.
\item
The inclusion $A \odot B \subseteq A \tensMin B$ has the ideal separation property, that is,  closed ideals $I_1,I_2 \subseteq A \tensMin B$ satisfy $I_1 \subseteq I_2$ whenever $I_1 \cap (A \odot B) \subseteq I_2 \cap (A \odot B)$.
\item
We have $I = \overline{I \cap (A \odot B)}$ for every closed ideal $I \subseteq A \tensMin B$.
\item
The inclusion $A \tensProj B \subseteq A \tensMin B$ has the ideal separation property. 
\item
We have $I = \overline{I \cap (A \tensProj B)}$ for every closed ideal $I \subseteq A \tensMin B$.
\item
We have $P = \overline{P \cap (A \tensProj B)}$ for every closed, prime ideal $P \subseteq A \tensMin B$.
\end{enumerate}
\end{thm}
\begin{proof}
The equivalence of~(1) and~(2) is shown in \cite[Proposition~2.16]{BlaKir04PureInf}.
The equivalence of~(3) and~(4) follows from \autoref{prp:CharISPDense} using that $A \odot B$ is dense in $A \otimes B$.
Similarly we see that~(5) and~(6) are equivalent.
It is also easy to see that(2) implies~(4), that~(4) implies~(6), and that~(6) implies~(7).
		
Let us show that~(7) implies~(1).
Given a closed, prime ideal $P \subseteq A \otimes B$, using the assumption $P = \overline{P \cap (A \tensProj B)}$, it follows from \autoref{prp:PrimeIdlTensMin} that $P = (J \tensMin B) + (A \tensMin K)$ for closed, prime ideals $I \subseteq A$ and $K \subseteq B$.
This verifies condition~(iii) in \cite[Proposition~2.16]{BlaKir04PureInf}, which then shows that the pair ($A$,$B$) satisfies Tomiyama's property~(F).
\end{proof}
	
\begin{rmk}
Consider the following statements for a closed ideal $I \subseteq A \tensMin B$:
\begin{enumerate}
\item
$I$ is the closure of the sum of all product ideals contained in $I$.
\item
$I = \overline{I \cap (A \odot B)}$.
\item
$I = \overline{I \cap (A \tensProj B)}$.
\end{enumerate}
If a closed ideal satisfies~(1), then it satisfies~(2), which in turn implies that it satisfies~(3).
Further, \autoref{prp:CharPropF} shows that if \emph{all} closed ideals in $A \tensMin B$ satisfy~(3), then the pair ($A$,$B$) satisfies Tomiyama's property~(F) and consequently all closed ideals in $A \tensMin B$ satisfy~(1) and~(2).

If the pair ($A$,$B$) does not satisfy Tomiyama's property~(F), then there necessarily exists a closed ideal in $A \tensMin B$ that does not satisfy~(3) (hence neither~(1), nor~(2)).
In this case, it remains unclear if there also exist closed ideals in $A \tensMin B$ that satisfy~(3) but not~(2), or that satisfy~(2) but not~(1).
\end{rmk}

Note that statement~(2) in \autoref{prp:AlgTP} below shows that, unlike the ISP, the IIP always holds for the inclusion of the algebraic into the minimal tensor product of \ca{s}.

\begin{pgr}
\label{pgr:SliceMaps}
Let $A$ and $B$ be \ca{s}.
Given $\psi \in B^*$, that is, a bounded, linear functional $\psi \colon B \to \CC$, we consider the left slice map $L_\psi^{(0)} \colon A \odot B \to A$ given by
\[
L_\psi^{(0)} \left( \sum_k a_k \otimes b_k \right) 
= \sum_k \psi(b_k)a_k.
\]

If $\psi$ is a positive, linear functional, then it is well-known that $L_\psi^{(0)}$ extends to the completion $A \tensMin B$ and induces a positive, linear map $L_\psi \colon A \tensMin B \to A$;
see \cite[Paragraph~II.9.7.1]{Bla06OpAlgs}.
Using that every bounded, linear functional on a \ca{} is a linear combination of positive, linear functionals (\cite[Theorem~II.6.3.4]{Bla06OpAlgs}), we see that $L_\psi^{(0)}$ extends to a bounded, linear map $L_\psi \colon A \tensMin B \to A$ for every $\psi \in B^*$.
\end{pgr}

\begin{thm}
\label{prp:AlgTP}
Let $A$ and $B$ be \ca{s} containing subalgebras $A_0 \subseteq A$ and $B_0 \subseteq B$.
The following statements hold:
\begin{enumerate}
\item
The inclusion $A_0 \odot B_0 \subseteq A \otimes B$ has the ISP if and only if $A_0 \subseteq A$ and $B_0 \subseteq B$ have the ISP and the pair ($A$,$B$) has Tomiyama's property~(F).
\item
The inclusion $A_0 \odot B_0 \subseteq A \otimes B$ has the IIP if and only if $A_0 \subseteq A$ and $B_0 \subseteq B$ have the IIP.
\end{enumerate}
\end{thm}
\begin{proof}
(1)
To show the forward implication, assume that $A_0 \odot B_0 \subseteq A \otimes B$ has the ISP.
Then $A \odot B \subseteq A \otimes B$ has the ISP (see \autoref{prp:Superalgebras}), and thus ($A$,$B$) has Tomiyama's property~(F) by \autoref{prp:CharPropF}.
To verify that $A_0 \subseteq A$ has the ISP, let $I,J \subseteq A$ be distinct closed ideals.
Then $I \otimes B$ and $J \otimes B$ are distinct closed ideals in $A \otimes B$.
By assumption, their intersections with $A_0 \odot B_0$ are also distinct.
Without loss of generality, we assume that there is $t \in A \tensMin B$ with
\[
t \in (I \tensMin B) \cap (A_0 \odot B_0), \andSep
t \notin J \tensMin B.
\]

Since ($A$,$B$) has Tomiyama's property~(F), the slice map problem has a positive solution with respect to closed ideals;
see \cite[Remark~24]{Was76SliceMapProblem}.
Thus, if an element $x \in A \otimes B$ satisfies $L_\psi(x) \in J$ for every $\psi \in B^*$, then $x \in J \otimes B$.
Since our element~$t$ does not belong to $J \otimes B$, we can pick $\psi \in B^*$ such that
\[
L_\psi(t) \notin J.
\]
We have $L_\psi(I \otimes B) \subseteq I$, and therefore $L_\psi(t) \in I$.
Finally, since $t \in A_0 \odot B_0$, say $t = \sum_k a_k \otimes b_k$ for a finite sum of elementary tensors with $a_k \in A_0$ and $b_k \in B_0$, we have 
\[
L_\psi(t)
= \sum_k \psi(b_k)a_k \in A_0.
\]
Thus, the element $L_\psi(t)$ belongs to $I \cap A_0$ but not to $J$, which shows that $A_0$ separates ideals in $A$.
Analogously one shows that $B_0 \subseteq B$ has the ISP.

\medskip

To show the backwards implication, assume that $A_0 \subseteq A$ and $B_0 \subseteq B$ have the ISP and that the pair ($A$,$B$) has Tomiyama's property~(F).
Given a subset $S$ in a \ca{}, let us use $\langle S \rangle$ to denote the closed ideal generated by $S$.
To show the ISP for $A_0 \odot B_0 \subseteq A \otimes B$, we verify condition~(3) in \autoref{prp:CharISPGeneral}, that is, we prove that $I = \langle I \cap (A_0 \odot B_0) \rangle$ for every closed ideal $I \subseteq A \otimes B$.
We first consider the case that $I = J \otimes K$ for closed ideals $J \subseteq A$ and $K \subseteq B$.
Note that
\[
(J \otimes K) \cap (A_0 \odot B_0) \supseteq (J \cap A_0) \odot (K \cap B_0).
\]
Since $A_0 \subseteq A$ has the ISP, we have $J = \langle J \cap A_0\rangle$ by \autoref{prp:CharISPGeneral}.
Similarly, we get $K = \langle K \cap B_0 \rangle$, and thus
\[
\big\langle (J \otimes K)\cap(A_0 \odot B_0) \big\rangle
\supseteq \big\langle (J \cap A_0) \odot (K \cap B_0) \big\rangle 
\supseteq \big\langle J \odot K \big\rangle 
= J \otimes K.
\]
		
Now, let $I \subseteq A \otimes B$ be an arbitrary closed ideal.
Since ($A$,$B$) has Tomiyama's property~(F), we obtain a family of product ideals $I_\lambda$ such that $I$ is the closure of $\sum_\lambda I_\lambda$.
For each $\lambda$, we have
\[
\big\langle I \cap (A_0 \odot B_0) \big\rangle
\supseteq \big\langle I_\lambda \cap (A_0 \odot B_0) \big\rangle
= I_\lambda,
\]
and thus
\[
\big\langle I \cap (A_0 \odot B_0) \big\rangle
\supseteq \big\langle \sum_\lambda I_\lambda \big\rangle
= I,
\]
as desired.

\medskip

(2)
To show the forward implication, assume that $A_0 \odot B_0 \subseteq A \otimes B$ has the IIP.
To verify that $A_0 \subseteq A$ has the IIP, let $I \subseteq A$ be a non-zero, closed ideal.
Then $I \otimes B$ is a non-zero, closed ideal in $A \otimes B$.
By assumption, we obtain a non-zero element $t \in (I \tensMin B) \cap (A_0 \odot B_0)$.
Since slice maps detect if an element in $A \otimes B$ is zero (see~S3 on p.541 in \cite{Was76SliceMapProblem}), we obtain $\psi \in B^*$ such that
\[
L_\psi(t) \neq 0.
\]
As in~(1), we get $L_\psi(t) \in I \cap A_0$.
Analogously one shows that $B_0 \subseteq B$ has the IIP.

\medskip

To show the backwards implication, assume that $A_0 \subseteq A$ and $B_0 \subseteq B$ have the IIP.
Let $I \subseteq A \tensMin B$ be a non-zero, closed ideal.
Then $I$ contains a nonzero elementary tensor $a \otimes b$;
see, for example, \cite[Lemma~2.12(ii)]{BlaKir04PureInf}.
Let $J$ be the closed ideal of $A$ generated by $a$, and analogously for $K$.
Then $J \otimes K \subseteq I$.
Since $A_0 \subseteq A$ has the IIP, we obtain a nonzero element $x \in J \cap A_0$.
Similarly, we obtain a nonzero element $y \in K \cap B_0$.
Then
\[
0 \neq x \otimes y 
\in (J \cap A_0) \odot (K \cap B_0)
\subseteq (J \odot K) \cap (A_0 \odot B_0)
\subseteq I \cap (A_0 \odot B_0)
\]
as desired.
\end{proof}

The following corollary is immediate by \autoref{prp:AlgTP}.

\begin{cor}
Let $A$ be a \ca{}, and suppose $A_0 \subseteq A$ is a subalgebra with the ISP (the IIP). 
Then $M_n(A_0) \subseteq M_n(A)$ has the ISP (the IIP).
\end{cor}

The next result recovers \cite[Corollary 2.3]{AleKye19UniqueCNormGpRg} and also proves the converse direction. 
\begin{cor}
Let $\Gamma$ and $\Delta$ be discrete groups.
Then $\Gamma \times \Delta$ is algebraically $C^*$-unique if and only if $\Gamma$ and $\Delta$ are.
\end{cor}
\begin{proof}
Similarly as in \autoref{prp:CharStarReg}, we see that a discrete group $H$ is algebraically $C^*$-unique if and only if $H$ is amenable and the inclusion $\CC[H] \subseteq C^*_\red(H)$ has the IIP.
The natural isomorphism between $C^*_\red( \Gamma\times\Delta )$ and $C^*_\red(\Gamma) \tensMin C^*_\red(\Delta)$ identifies $\CC[ \Gamma\times\Delta ]$ with $\CC[\Gamma] \odot \CC[\Delta]$, as shown in the following commutative diagram:
\[
\xymatrix{
\CC[ \Gamma\times\Delta ]\ar@{^{(}->}[r] \ar[d]_{\cong}
& C^*_\red( \Gamma\times\Delta ) \ar[d]^{\cong} \\
\CC[\Gamma] \odot \CC[\Delta] \ar@{^{(}->}[r]
& C^*_\red(\Gamma) \tensMin C^*_\red(\Delta).
}
\]
Now the result follows from \autoref{prp:AlgTP} together with the result that $\Gamma\times\Delta$ is amenable if and only if $\Gamma$ and $\Delta$ are amenable (\cite[Proposition~1.13]{Pat88Amen}).
\end{proof}

\begin{exa}
Let $\Gamma = \ZZ[\tfrac{1}{pq}]\rtimes\ZZ^2$ be one of the torsion-free, algebraically $C^*$-unique groups constructed by Scarparo in \cite{Sca20TorsFreeAlgCUniqueGp}, and let $\Delta$ be any locally finite group.
Then $\Lambda \times \Delta$ is algebraically $C^*$-unique.
\end{exa}

\begin{thm}
\label{prp:ProjTP}
Let $\iota_A \colon A_0 \to A$ and $\iota_B \colon B_0 \to B$ be bounded homomorphisms from Banach algebras $A_0$ and $B_0$ to \ca{s} $A$ and $B$, and let $\iota \colon A_0 \tensProj B_0 \to A \tensMin B$ be the naturally induced map.
The following statements hold:
\begin{enumerate}
\item
The inclusion $\iota(A_0 \tensProj B_0) \subseteq A \tensMin B$ has the ISP if and only if $\iota_A(A_0) \subseteq A$ and $\iota_B(B_0) \subseteq B$ have the ISP and the pair ($A$,$B$) has Tomiyama's property~(F).
\item
The inclusion $\iota(A_0 \tensProj B_0) \subseteq A \tensMin B$ has the IIP if and only if $\iota_A(A_0) \subseteq A$ and $\iota_B(B_0) \subseteq B$ have the IIP.
\end{enumerate}
\end{thm}
\begin{proof}
Given $\psi \in B^*$, we consider the left slice map $L_\psi \colon A \otimes B \to A$ as in \autoref{pgr:SliceMaps}.
It follows easily from the universal property of the projective tensor product that there exists a bounded, linear map $\widehat{L}_{\psi} \colon A_0 \tensProj B_0 \to A_0$ determined on finite sums of elementary tensors by
\[
\widehat{L}_{\psi} \left( \sum_k a_k \otimes b_k \right) 
= \sum_k \psi(\iota_B(b_k))a_k.
\]

Using that $L_\psi\circ\iota$ and $\iota_A\circ \widehat{L}_{\psi}$ are bounded and linear and agree on elementary tensors, we see that both maps agree.
Thus, the following diagram commutes:
\[
\xymatrix{
A_0 \tensProj B_0 \ar[r]^{\iota} \ar[d]_{\widehat{L}_{\psi}}
& A \tensMin B \ar[d]^{L_\psi} \\
A_0 \ar[r]_{\iota_A}
& A.
}
\]
We will use that if $I \subseteq A$ is a closed ideal, then $L_{\psi}(I \tensMin B) \subseteq I$.

\medskip

(1)
To show the forward implication, assume that $\iota(A_0 \tensProj B_0) \subseteq A \tensMin B$ has the ISP.
Since $\iota(A_0 \tensProj B_0)$ is contained in $A \tensProj B$, it follows that $A \tensProj B \subseteq A \tensMin B$ has the ISP (see \autoref{prp:Superalgebras}), and thus ($A$,$B$) has Tomiyama's property~(F) by \autoref{prp:CharPropF}.
To verify that $\iota_A(A_0) \subseteq A$ has the ISP, let $I,J \subseteq A$ be distinct closed ideals.
Then $I \otimes B$ and $J \otimes B$ are distinct closed ideals in $A \otimes B$.
By assumption, their intersections with $\iota(A_0 \tensProj B_0)$ are also distinct.
Without loss of generality, we assume that there is $t \in A_0 \tensProj B_0$ with
\[
\iota(t) \in I \tensMin B, \andSep
\iota(t) \notin J \tensMin B.
\]

As in the proof of \autoref{prp:AlgTP}, we pick $\psi \in B^*$ such that
\[
L_\psi(\iota(t)) \notin J.
\]

Set $x := \widehat{L}_{\psi}(t) \in A_0$.
Then
\[
\iota_A(x) 
= \iota_A(\widehat{L}_{\psi}(t))
= L_{\psi}(\iota(t))
\in L_{\psi}(I \otimes B)
\subseteq I, \andSep
\iota_A(x) 
= L_{\psi}(\iota(t))
\notin J .
\]
Thus, $\iota_A(x)$ separates $I$ and $J$.
Analogously, one shows that $\iota_B(B_0) \subseteq B$ has the ISP.

\medskip

To show the backwards implication, assume that $\iota_A(A_0) \subseteq A$ and $\iota_B(B_0) \subseteq B$ have the ISP and that the pair ($A$,$B$) has Tomiyama's property~(F).
By \autoref{prp:AlgTP}, it follows that
\[
\iota_A(A_0) \odot \iota_B(B_0) \subseteq A \tensMin B
\]
has the ISP.
Since $\iota_A(A_0) \odot \iota_B(B_0)$ is contained in $\iota(A_0 \tensProj B_0)$, it follows from \autoref{prp:Superalgebras} that $\iota(A_0 \tensProj B_0) \subseteq A \tensMin B$ has the ISP.

\medskip

(2)
To show the forward implication, assume that $\iota(A_0 \tensProj B_0) \subseteq A \tensMin B$ has the IIP.
To show that $\iota_A(A_0) \subseteq A$ has the IIP, let $I \subseteq A$ be a non-zero, closed ideal.
Then $I \tensMin B \subseteq A \tensMin B$ is a non-zero, closed ideal, and by assumption we obtain $t \in A_0 \tensProj B_0$ such that $0 \neq \iota(t) \in I \tensMin B$.
As in the proof of \autoref{prp:AlgTP}, we can pick $\psi \in B^*$ such that
\[
L_\psi(\iota(t)) \neq 0.
\]
Set $x := \widehat{L}_{\psi}(t) \in A_0$.
As in~(1), we get $\iota_A(x) \in I$.
Further, we have 
\[
\iota_A(x) 
= \iota_A(\widehat{L}_{\psi}(t))
= L_{\psi}(\iota(t))
\neq 0,
\]
and thus $I \cap \iota_A(A_0) \neq \{0\}$.
Analogously, one shows that $\iota_B(B_0) \subseteq B$ has the IIP.

\medskip

The backwards implication follows from \autoref{prp:AlgTP}, similarly as in~(1).
\end{proof}

We are now ready to prove the main result of this section, characterizing when the Cartesian product of two locally compact groups has the ISP or the IIP.
We discuss in \autoref{rmk:ProductGp} how our result generalizes the characterization of $*$-regularity and $C^*$-uniqueness of product groups from \cite[Corollary~3.5]{HauKanVoi90StarRegTensProdCAlg}.

\begin{thm}
\label{prp:ProductGp}
Let $G$ and $H$ be locally compact groups.
The following statements hold:
\begin{enumerate}
\item
The product $G \times H$ has the ISP if and only if $G$ and $H$ have the ISP and the pair of \ca{s} $C^*_\red(G)$ and~$C^*_\red(H)$ has Tomiyama's property~(F).
\item
The product $G \times H$ has the IIP if and only if $G$ and $H$ have the IIP.
\end{enumerate}
\end{thm}
\begin{proof}
Using the natural isomorphisms of $L^1( G \times H )$ with $L^1(G) \tensProj L^1(H)$ (see, for example, \cite[Corollary~1.10.14]{Pal94BAlg1}), and of $C^*_\red( G \times H )$ with $C^*_\red(G) \tensMin C^*_\red(H)$, we obtain the following commutative diagram:
\[
\xymatrix{
L^1(G \times H) \ar@{^{(}->}[r] \ar[d]_{\cong}
& C^*_\red(G \times H) \ar[d]^{\cong} \\
L^1(G) \tensProj L^1(H) \ar[r]
& C^*_\red(G) \tensMin C^*_\red(H).
}
\]

It follows that the natural map $L^1(G) \tensProj L^1(H) \to C^*_\red(G) \tensMin C^*_\red(H)$ is an inclusion, and we use it to identify $L^1(G) \tensProj L^1(H)$ with a dense subalgebra of $C^*_\red(G) \tensMin C^*_\red(H)$.

Using this, we see that $G \times H$ has the ISP (the IIP) if and only if the inclusion $L^1(G) \tensProj L^1(H) \subseteq C^*_\red(G) \tensMin C^*_\red(H)$ has the ISP (the IIP).
Now the result follows from \autoref{prp:ProjTP}. 
\end{proof}

Following Kirchberg and Wassermann \cite{KirWas99ExactGps}, a locally compact group $G$ is said to be \emph{exact} if taking reduced crossed products by $G$ preserves short exact sequences of \ca{s} with continuous $G$-actions.
By considering trivial actions on \ca{s} in a short exact sequence, we see that $C^*_\red(G)$ is an exact \ca{} whenever $G$ is an exact, locally compact group.
Conversely, a discrete group $G$ is exact if (and only if) $C^*_\red(G)$ is exact (\cite[Theorem~5.2]{KirWas99ExactGps}), and it remains a major open problem if this also holds for nondiscrete groups;
see \cite[Problem~9.3]{Ana02AmenExactDynSysCa} and \cite{Man21ExactVsCExact}.
	
\begin{cor}
\label{prp:ISP-PrductExactGp}
Let $G$ and $H$ be locally compact groups and assume that at least one of them is exact.
Then $G \times H$ has the ISP if and only if $G$ and $H$ have the ISP.
\end{cor}
\begin{proof}
A pair of \ca{s} automatically satisfies Tomiyama's property~(F) if one of the \ca{s} is exact. {To see this, note that \cite[Proposition~2.17(2)]{BlaKir04PureInf} implies that \cite[Proposition~2.16(ii)]{BlaKir04PureInf} is satisfied, and the statement is now a consequence of \autoref{prp:CharPropF}. }
Thus, it follows from the assumption that the pair of \ca{s} $C^*_\red(G)$ and~$C^*_\red(H)$ satisfies Tomiyama's property~(F).
Now the result follows from \autoref{prp:ProductGp}.
\end{proof}

\begin{rmk}
\label{rmk:ProductGp}
Let $G$ and $H$ be locally compact groups.
By \cite[Corollary~3.5]{HauKanVoi90StarRegTensProdCAlg}, the product $G \times H$ is $*$-regular ($C^*$-unique) if and only $G$ and $H$ are $*$-regular ($C^*$-unique).

\autoref{prp:ProductGp} recovers and generalizes these results.
Indeed, it follows from \cite[Proposition~1.13]{Pat88Amen} that $G \times H$ is amenable if and only if $G$ and $H$ are amenable.
Therefore, \cite[Corollary~3.5]{HauKanVoi90StarRegTensProdCAlg} follows from \autoref{prp:ProductGp} (and \autoref{prp:ISP-PrductExactGp}) using that by \autoref{prp:CharStarReg} a locally compact group is $*$-regular ($C^*$-unique) if and only if it is amenable and satisfies the ISP (the IIP).
\end{rmk}
	
With view towards \autoref{prp:ProductGp}, our \autoref{qst:DiscreteISP} if every discrete group has the ISP is closely related to \autoref{qst:PropFGroupAlgs} below.
Indeed, if $\Gamma_1$ and $\Gamma_2$ are discrete groups such that the pair of \ca{s} $C^*_\red(\Gamma_1)$ and~$C^*_\red(\Gamma_2)$ does not have Tomiyama's property (F), then $\Gamma_1 \times \Gamma_2$ does not have the ISP.
Note that such groups are necessarily non-exact.
They also cannot both be $C^*$-simple.
	
\begin{qst}
\label{qst:PropFGroupAlgs}
Are there discrete groups $\Gamma_1$ and $\Gamma_2$ such that the pair of \ca{s} $C^*_\red(\Gamma_1)$ and~$C^*_\red(\Gamma_2)$ does not have Tomiyama's property (F)?
\end{qst}


\begin{thebibliography}{KKL{\etalchar{+}}21}

\bibitem[AK19]{AleKye19UniqueCNormGpRg}
\bgroup\scshape{}V.~Alekseev\egroup{} and \bgroup\scshape{}D.~Kyed\egroup{},
  Uniqueness questions for {$C^*$}-norms on group rings,  \emph{Pacific J.
  Math.} \textbf{298} (2019), 257--266.

\bibitem[AD02]{Ana02AmenExactDynSysCa}
\bgroup\scshape{}C.~Anantharaman-Delaroche\egroup{}, Amenability and exactness
  for dynamical systems and their \ca{s},  \emph{Trans. Amer. Math. Soc.}
  \textbf{354} (2002), 4153--4178.

\bibitem[Aus21]{Aus21SpecInvTwConvAlg}
\bgroup\scshape{}A.~Austad\egroup{}, Spectral invariance of
  {$*$}-representations of twisted convolution algebras with applications in
  {G}abor analysis,  \emph{J. Fourier Anal. Appl.} \textbf{27} (2021), Paper
  No.~56, 22p.

\bibitem[AO22]{AusOrt22CUniuqueGpds}
\bgroup\scshape{}A.~Austad\egroup{} and \bgroup\scshape{}E.~Ortega\egroup{},
  {$C^*$}-uniqueness results for groupoids,  \emph{Int. Math. Res. Not. IMRN}
  (2022), 3057--3073.

\bibitem[AR23]{AusRau23arX:DetectIdlsRedCrProd}
\bgroup\scshape{}A.~Austad\egroup{} and \bgroup\scshape{}S.~Raum\egroup{},
  Detecting ideals in reduced crossed product \ca{s} of topological dynamical
  systems, preprint (2301.01027 [math.OA]), 2023.

\bibitem[Bar83]{Bar83StarRegUniqueCNorm}
\bgroup\scshape{}B.~A. Barnes\egroup{}, The properties {$\sp{\ast}
  $}-regularity and uniqueness of {$C\sp{\ast} $}-norm in a general \ca{},
  \emph{Trans. Amer. Math. Soc.} \textbf{279} (1983), 841--859.

\bibitem[Bla06]{Bla06OpAlgs}
\bgroup\scshape{}B.~Blackadar\egroup{}, \emph{Operator algebras},
  \emph{Encyclopaedia of Mathematical Sciences} \textbf{122}, Springer-Verlag,
  Berlin, 2006, Theory of \ca{s} and von Neumann algebras, Operator Algebras
  and Non-commutative Geometry, III.

\bibitem[BK04]{BlaKir04PureInf}
\bgroup\scshape{}E.~Blanchard\egroup{} and
  \bgroup\scshape{}E.~Kirchberg\egroup{}, Non-simple purely infinite \ca{s}:
  the {H}ausdorff case,  \emph{J. Funct. Anal.} \textbf{207} (2004), 461--513.

\bibitem[Ble88]{Ble88GeomTensProdCAlg}
\bgroup\scshape{}D.~P. Blecher\egroup{}, Geometry of the tensor product of
  \ca{s},  \emph{Math. Proc. Cambridge Philos. Soc.} \textbf{104} (1988),
  119--127.

\bibitem[Boi82a]{Boi82CtdGpsPolyDual}
\bgroup\scshape{}J.~Boidol\egroup{}, Connected groups with polynomially induced
  dual,  \emph{J. Reine Angew. Math.} \textbf{331} (1982), 32--46.

\bibitem[Boi82b]{Boi82StarRegSomeSolvGps}
\bgroup\scshape{}J.~Boidol\egroup{}, {$\sp{\ast} $}-regularity of some classes
  of solvable groups,  \emph{Math. Ann.} \textbf{261} (1982), 477--481.

\bibitem[BLSV78]{BoiLepSchVah78PrimGpCAlg}
\bgroup\scshape{}J.~Boidol\egroup{}, \bgroup\scshape{}H.~Leptin\egroup{},
  \bgroup\scshape{}J.~Sch\"{u}rman\egroup{}, and
  \bgroup\scshape{}D.~Vahle\egroup{}, R\"{a}ume primitiver {I}deale von
  {G}ruppenalgebren,  \emph{Math. Ann.} \textbf{236} (1978), 1--13.

\bibitem[Boi84]{Boi84GpAlgsUniqueCNorm}
\bgroup\scshape{}J.~Boidol\egroup{}, Group algebras with a unique {$C\sp{\ast}
  $}-norm,  \emph{J. Funct. Anal.} \textbf{56} (1984), 220--232.

\bibitem[BCS24]{BriCarSim24IdlStrEtaleGpd}
\bgroup\scshape{}K.~A. Brix\egroup{}, \bgroup\scshape{}T.~M. Carlsen\egroup{},
  and \bgroup\scshape{}A.~Sims\egroup{}, Some results regarding the ideal
  structure of \ca{s} of \'etale groupoids,  \emph{J. Lond. Math. Soc. (2)}
  \textbf{109} (2024), Paper No. e12870, 20.

\bibitem[Flo24]{Flo24arX:PolyGrwthFctlCalc}
\bgroup\scshape{}F.~I. Flores\egroup{}, Polynomial growth and functional
  calculus in algebras of integrable cross-sections, preprint
  (arXiv:2401.09730), 2024.

\bibitem[GKT23]{GarKitThi23arX:SemiprimeIdls}
\bgroup\scshape{}E.~Gardella\egroup{}, \bgroup\scshape{}K.~Kitamura\egroup{},
  and \bgroup\scshape{}H.~Thiel\egroup{}, Semiprime ideals in \ca{s}, J. Eur.
  Math. Soc. (to appear), preprint (arXiv:2311.17480), 2023.

\bibitem[GU23]{GefUrs23arX:SimpleCrProdFCHyperGps}
\bgroup\scshape{}S.~Geffen\egroup{} and \bgroup\scshape{}D.~Ursu\egroup{},
  Simplicity of crossed products by {FC}-hypercentral groups, preprint
  (2304.07852 [math.OA]), 2023.

\bibitem[GMR18]{GriMusRor18JustInfCAlgs}
\bgroup\scshape{}R.~Grigorchuk\egroup{}, \bgroup\scshape{}M.~Musat\egroup{},
  and \bgroup\scshape{}M.~R{\o}rdam\egroup{}, Just-infinite \ca{s},
  \emph{Comment. Math. Helv.} \textbf{93} (2018), 157--201.

\bibitem[GJ20]{GupJai20ProjTensProjCAlg}
\bgroup\scshape{}V.~P. Gupta\egroup{} and \bgroup\scshape{}R.~Jain\egroup{}, On
  {B}anach space projective tensor product of \ca{s},  \emph{Banach J. Math.
  Anal.} \textbf{14} (2020), 524--538.

\bibitem[Haa85]{Haa85GrothendieckIneq}
\bgroup\scshape{}U.~Haagerup\egroup{}, The {G}rothendieck inequality for
  bilinear forms on \ca{s},  \emph{Adv. in Math.} \textbf{56} (1985), 93--116.

\bibitem[HKV90]{HauKanVoi90StarRegTensProdCAlg}
\bgroup\scshape{}W.~Hauenschild\egroup{}, \bgroup\scshape{}E.~Kaniuth\egroup{},
  and \bgroup\scshape{}A.~Voigt\egroup{}, {$*$}-regularity and uniqueness of
  {$C^*$}-norm for tensor products of \ca{s},  \emph{J. Funct. Anal.}
  \textbf{89} (1990), 137--149.

\bibitem[JK11]{JaiKum11IdlsOpSpProjTPCAlg}
\bgroup\scshape{}R.~Jain\egroup{} and \bgroup\scshape{}A.~Kumar\egroup{},
  Ideals in operator space projective tensor product of \ca{s},  \emph{J. Aust.
  Math. Soc.} \textbf{91} (2011), 275--288.

\bibitem[KT90]{KawTom90PropTopDynSysCAlg}
\bgroup\scshape{}S.~o. Kawamura\egroup{} and
  \bgroup\scshape{}J.~Tomiyama\egroup{}, Properties of topological dynamical
  systems and corresponding \ca{s},  \emph{Tokyo J. Math.} \textbf{13} (1990),
  251--257.

\bibitem[KKL{\etalchar{+}}21]{KenKimLiRauUrs21arX:ISPEssGpd}
\bgroup\scshape{}M.~Kennedy\egroup{}, \bgroup\scshape{}S.-J. Kim\egroup{},
  \bgroup\scshape{}X.~Li\egroup{}, \bgroup\scshape{}S.~Raum\egroup{}, and
  \bgroup\scshape{}D.~Ursu\egroup{}, The ideal intersection property for
  essential groupoid \ca{s}, preprint (arXiv:2107.03980 [math.OA]), 2021.

\bibitem[KS19]{KenSch19NCBoundariesIdealRedCrProd}
\bgroup\scshape{}M.~Kennedy\egroup{} and
  \bgroup\scshape{}C.~Schafhauser\egroup{}, Noncommutative boundaries and the
  ideal structure of reduced crossed products,  \emph{Duke Math. J.}
  \textbf{168} (2019), 3215--3260.

\bibitem[KW99]{KirWas99ExactGps}
\bgroup\scshape{}E.~Kirchberg\egroup{} and
  \bgroup\scshape{}S.~Wassermann\egroup{}, Exact groups and continuous bundles
  of \ca{s},  \emph{Math. Ann.} \textbf{315} (1999), 169--203.

\bibitem[Lam01]{Lam01FirstCourse2ed}
\bgroup\scshape{}T.~Y. Lam\egroup{}, \emph{A first course in noncommutative
  rings}, second ed., \emph{Graduate Texts in Mathematics} \textbf{131},
  Springer-Verlag, New York, 2001.

\bibitem[Laz10]{Laz10IdlsMinTensProdCAlgs}
\bgroup\scshape{}A.~J. Lazar\egroup{}, The space of ideals in the minimal
  tensor product of \ca{s},  \emph{Math. Proc. Cambridge Philos. Soc.}
  \textbf{148} (2010), 243--252.

\bibitem[Leu11]{Leu11RegGenGpAlg}
\bgroup\scshape{}C.-W. Leung\egroup{}, {$\ast$}-regularity of certain
  generalized group algebras,  \emph{Arch. Math. (Basel)} \textbf{96} (2011),
  445--454.

\bibitem[LN04]{LeuNg04PermPropCUniqueGps}
\bgroup\scshape{}C.-W. Leung\egroup{} and \bgroup\scshape{}C.-K. Ng\egroup{},
  Some permanence properties of {$C^*$}-unique groups,  \emph{J. Funct. Anal.}
  \textbf{210} (2004), 376--390.

\bibitem[LN06a]{LeuNg06FctlIdStarReg1}
\bgroup\scshape{}C.-W. Leung\egroup{} and \bgroup\scshape{}C.-K. Ng\egroup{},
  Functional calculus and {$\ast$}-regularity of a class of {B}anach algebras,
  \emph{Proc. Amer. Math. Soc.} \textbf{134} (2006), 755--763.

\bibitem[LN06b]{LeuNg06FctlIdStarReg2}
\bgroup\scshape{}C.-W. Leung\egroup{} and \bgroup\scshape{}C.-K. Ng\egroup{},
  Functional calculus and {$\ast$}-regularity of a class of {B}anach algebras.
  {II},  \emph{J. Math. Anal. Appl.} \textbf{322} (2006), 699--711.

\bibitem[Lip72]{Lip72RepThyAlmCtdGps}
\bgroup\scshape{}R.~L. Lipsman\egroup{}, Representation theory of almost
  connected groups,  \emph{Pacific J. Math.} \textbf{42} (1972), 453--467.

\bibitem[Man21]{Man21ExactVsCExact}
\bgroup\scshape{}N.~Manor\egroup{}, Exactness versus {$\rm C^*$}-exactness for
  certain non-discrete groups,  \emph{Integral Equations Operator Theory}
  \textbf{93} (2021), Paper No. 20, 13.

\bibitem[MZ55]{MonZip55TopTransfGps}
\bgroup\scshape{}D.~Montgomery\egroup{} and
  \bgroup\scshape{}L.~Zippin\egroup{}, \emph{Topological transformation
  groups}, Interscience Publishers, New York-London, 1955.

\bibitem[Osi02]{Osi02ElemClassesGps}
\bgroup\scshape{}D.~V. Osin\egroup{}, Elementary classes of groups,  \emph{Mat.
  Zametki} \textbf{72} (2002), 84--93.

\bibitem[Pal94]{Pal94BAlg1}
\bgroup\scshape{}T.~W. Palmer\egroup{}, \emph{Banach algebras and the general
  theory of {$\sp *$}-algebras. {V}ol. {I}}, \emph{Encyclopedia of Mathematics
  and its Applications} \textbf{49}, Cambridge University Press, Cambridge,
  1994, Algebras and Banach algebras.

\bibitem[Pal01]{Pal01BAlg2}
\bgroup\scshape{}T.~W. Palmer\egroup{}, \emph{Banach algebras and the general
  theory of {$*$}-algebras. {V}ol. 2}, \emph{Encyclopedia of Mathematics and
  its Applications} \textbf{79}, Cambridge University Press, Cambridge, 2001,
  $*$-algebras.

\bibitem[Pat88]{Pat88Amen}
\bgroup\scshape{}A.~L.~T. Paterson\egroup{}, \emph{Amenability},
  \emph{Mathematical Surveys and Monographs} \textbf{29}, American Mathematical
  Society, Providence, RI, 1988.

\bibitem[Pis20]{Pis20NonNuclWEP-LLP}
\bgroup\scshape{}G.~Pisier\egroup{}, A non-nuclear \ca{} with the weak
  expectation property and the local lifting property,  \emph{Invent. Math.}
  \textbf{222} (2020), 513--544.

\bibitem[RS00]{ReiSte00HarmAna}
\bgroup\scshape{}H.~Reiter\egroup{} and \bgroup\scshape{}J.~D.
  Stegeman\egroup{}, \emph{Classical harmonic analysis and locally compact
  groups}, second ed., \emph{London Mathematical Society Monographs. New
  Series} \textbf{22}, The Clarendon Press, Oxford University Press, New York,
  2000.

\bibitem[Rya02]{Rya02IntroTensProdBSp}
\bgroup\scshape{}R.~A. Ryan\egroup{}, \emph{Introduction to tensor products of
  {B}anach spaces}, \emph{Springer Monographs in Mathematics}, Springer-Verlag
  London, Ltd., London, 2002.

\bibitem[Sca20]{Sca20TorsFreeAlgCUniqueGp}
\bgroup\scshape{}E.~Scarparo\egroup{}, A torsion-free algebraically {$\rm
  C^*$}-unique group,  \emph{Rocky Mountain J. Math.} \textbf{50} (2020),
  1813--1815.

\bibitem[Sie10]{Sie10IdealsRedCrProd}
\bgroup\scshape{}A.~Sierakowski\egroup{}, The ideal structure of reduced
  crossed products,  \emph{M\"{u}nster J. Math.} \textbf{3} (2010), 237--261.

\bibitem[Tao14]{Tao14Hilb5thPbm}
\bgroup\scshape{}T.~Tao\egroup{}, \emph{Hilbert's fifth problem and related
  topics}, \emph{Graduate Studies in Mathematics} \textbf{153}, American
  Mathematical Society, Providence, RI, 2014.

\bibitem[TSH98]{TatShiHir98IndLimTopGps}
\bgroup\scshape{}N.~Tatsuuma\egroup{}, \bgroup\scshape{}H.~Shimomura\egroup{},
  and \bgroup\scshape{}T.~Hirai\egroup{}, On group topologies and unitary
  representations of inductive limits of topological groups and the case of the
  group of diffeomorphisms,  \emph{J. Math. Kyoto Univ.} \textbf{38} (1998),
  551--578.

\bibitem[Tom67]{Tom67FubiniTensProdCAlgs}
\bgroup\scshape{}J.~Tomiyama\egroup{}, Applications of {F}ubini type theorem to
  the tensor products of \ca{s},  \emph{Tohoku Math. J. (2)} \textbf{19}
  (1967), 213--226.

\bibitem[Was76]{Was76SliceMapProblem}
\bgroup\scshape{}S.~Wassermann\egroup{}, The slice map problem for \ca{s},
  \emph{Proc. London Math. Soc. (3)} \textbf{32} (1976), 537--559.

\end{thebibliography}

\providecommand{\etalchar}[1]{$^{#1}$}
\providecommand{\bysame}{\leavevmode\hbox to3em{\hrulefill}\thinspace}
\providecommand{\noopsort}[1]{}
\providecommand{\mr}[1]{\href{http://www.ams.org/mathscinet-getitem?mr=#1}{MR~#1}}
\providecommand{\zbl}[1]{\href{http://www.zentralblatt-math.org/zmath/en/search/?q=an:#1}{Zbl~#1}}
\providecommand{\jfm}[1]{\href{http://www.emis.de/cgi-bin/JFM-item?#1}{JFM~#1}}
\providecommand{\arxiv}[1]{\href{http://www.arxiv.org/abs/#1}{arXiv~#1}}
\providecommand{\doi}[1]{\url{http://dx.doi.org/#1}}
\providecommand{\MR}{\relax\ifhmode\unskip\space\fi MR }
\providecommand{\MRhref}[2]{%
  \href{http://www.ams.org/mathscinet-getitem?mr=#1}{#2}
}
\providecommand{\href}[2]{#2}

\end{document}